\documentclass[12pt]{amsart}
\usepackage{amsmath}
\usepackage{amssymb}
\usepackage{amsthm}
\usepackage{enumerate}
\usepackage[mathscr]{eucal}
\usepackage{adjustbox}
\usepackage{float}
\usepackage{longtable}
\usepackage{lipsum}
\usepackage{array}
\usepackage{array}
\usepackage{mathtools}
\usepackage{amssymb}
\usepackage{mathtools}
\usepackage{caption}
\usepackage{fullpage}
\usepackage{makecell}

\DeclareMathOperator{\ord}{ord}
\DeclareMathAlphabet{\mathpzc}{OT1}{pzc}{m}{it}
\DeclareMathOperator{\Div}{div}

\newcommand{\ghl}{\left \lvert }
\newcommand{\ghr}{\right\rvert}
\newcommand{\Br}{\left \lvert B \right\rvert}

\def\mod#1{{\ifmmode\text{\rm\ (mod~$#1$)}
\else\discretionary{}{}{\hbox{ }}\rm(mod~$#1$)\fi}}

\usepackage{float}
\restylefloat{table}

\numberwithin{equation}{section}

\def\mod#1{{\ifmmode\text{\rm\ (mod~$#1$)}
\else\discretionary{}{}{\hbox{ }}\rm(mod~$#1$)\fi}}

\usepackage{hyperref}

\usepackage{amssymb}

\usepackage{enumitem}

\usepackage{graphicx}

\makeatletter
\@namedef{subjclassname@2020}{%
  \textup{2020} Mathematics Subject Classification}
\makeatother

\usepackage[T1]{fontenc}
\newtheorem{theorem}{Theorem}[section]
\newtheorem{corollary}[theorem]{Corollary}
\newtheorem{lemma}[theorem]{Lemma}
\newtheorem{proposition}[theorem]{Proposition}

\newtheorem{conjecture}[theorem]{Conjecture}

\theoremstyle{definition}

\newtheorem{remark}[theorem]{Remark}

\numberwithin{equation}{section}

\begin{document}

%%%%% To ease editing, for IMPAN journals add:

\baselineskip=17pt

%%%%%%%%%%%%%%%%

\title[Multiples of integral points on Mordell curves]{Multiples of integral points on Mordell curves}

\author[A. Ghadermarzi]{Amir Ghadermarzi}
\address{School of Mathematics, Statistics and Computer Science, College of Science, University of Tehran, Tehran, Iran \\
 \&  School of Mathematics, Institute of Research in Fundamental Science (IPM), P.O.BOX 19395-5746 }
 \email {a.ghadermarzi@ut.ac.ir}

\date{}

\begin{abstract}
Let $B$ be a sixth-power-free integer and $P$ be a non-torsion  point on the  Mordell curve $E_B:y^2=x^3+B$. In this paper, we study integral multiples $[n]P$ of $P$. Among other results, we show that $P$ has at most three integral multiples with $n>1$. This result is sharp in the sense that there are points $P$ with exactly three integral multiples $[n]P$ and $n>1$. As an application, we discuss the number of integral points on the quasi-minimal model of rank 1 Mordell curves.
\end{abstract}



\maketitle

\section{Introduction}
\label{intro}
A well-known theorem of Siegel states that there are only finitely many integral points on an elliptic curve $E/\mathbb{Q}$. Therefore, it is clear that if $P \in E(\mathbb{Q})$ is a non-torsion point, then there are at most finitely many integral points among multiples $[n]P$ of $P$. It is an easy exercise to show that if $P$ is a non-torsion, non-integral point on an elliptic curve defined over $\mathbb{Z}$, it has no integral multiples. The natural question would be: is there a bound on the number of integral multiples of a point on a given elliptic curve? One can keep rescaling the coordinates of a point on the curve to construct more and more integral points. This artificial construction can be avoided if we only consider minimal curves. In the case of minimal curves, it seems plausible to assume the existence of a uniform bound. Indeed, Lang \cite[p.~140]{Lang} conjectured that the number of integral points on a quasi-minimal model of an elliptic curve $E(\mathbb{Q})$ can be bounded above by the rank of the curve. Hindry and Silverman \cite{Sil1} proved Lang's conjecture by assuming Szpiro's conjecture. Silverman \cite [Theorem A]{Sil2} proved this result unconditionally for curves with integral $j$-invariant, and more generally for the curves with $j$-invariant non-integral for at most a fixed number of primes. Gross and Silverman \cite{GrSi} obtained an explicit version of Silverman's result. Based on their work, as a special case, the number of integral points on any rank 1 quasi-minimal elliptic curve $E/\mathbb{Q}$ with integral $j$-invariant is bounded by $3.28 \times 10^{33}$. We use the method of Ingram \cite{ingram} to obtain much better results for Mordell curves of rank 1.
 
The methods utilizing lower bounds for linear forms in elliptic logarithms give effective bounds on the value of $n$ such that $[n]P$ is integral, but the bounds depend on the height of the curve and the point $P$. Ingram used division polynomials of elliptic curves to make this dependence more explicit. Then he applied a gap principle to show that a constant $C$ exists such that at most one multiple $[n]P$ is integral for $n>C$. Stange \cite{stange} modified Ingram's argument to obtain a bound that depends on the height ratio $ h(E)/\hat{h} (P) $. Note that Lang \cite[p.~92]{Lang} conjectured that there is a uniform constant such that for all elliptic curves $E/\mathbb{Q}$ in minimal Weierstrass form, the ratio $ h(E)/\hat{h} (P) $ is bounded above by that constant. Lang's conjecture is known to be true for elliptic curves with integral $j$-invariant or the ones in which the denominator of the $j$-invariant is divisible by a fixed number of primes \cite{Sil1,Sil3}.

In this paper, we consider the question of integral multiples of points on the Mordell curves $E_B: y^2=x^3+B$. To avoid the artificial construction of many integral multiples by rescaling, we consider the Mordell curves in quasi-minimal Weierstrass equation. Let $E/ \mathbb{Q}$ be an elliptic curve. A short Weierstrass equation for $E$ is called a quasi-minimal(Weierstrass) equation of $E$ if the discriminant of $E$ is minimal, subject to the condition of $E$ being in short Weierstrass form with integral coefficients. These models are minimal at all primes $p>3$. In the case of Mordell curves, the equation $y^2=x^3+B$ is a quasi-minimal model if and only if $B$ is a sixth-power-free integer. Note that the equation $y^2=x^3+B$ is a global minimal Weierstrass equation of $E$ if and only if $B$ is a sixth-power-free integer and $B \not \equiv 16 \pmod {64}$. If $B \equiv 16 \pmod {64}$, it is minimal at any odd prime $P$ \cite[Lemma 4.7]{vou} . Throughout, we assume that $B$ is a sixth-power-free integer, and we call the short Weierstrass form $y^2=x^3+B$ a  quasi-minimal Mordell curve. For a non-torsion, integral point $P$ on the quasi-minimal curve $E_B$, we would like to bound the number of $n$'s such that $[n]P$ is integral. 

In Section \ref{sec3}, we investigate small values of $n$ such that $[n]P$ is integral. We use some elementary arguments, $p$-adic valuation of the division polynomials, height estimations, and known results on the primitive divisors of elliptic divisibility sequences to prove the following theorem:

\begin{theorem} \label{main1}
Let $P$ be a non-torsion integral point on the quasi-minimal curve $E_B$, and $[n]P$ be integral. If $n>5$, then $n$ has no prime factor less than 11. Save for some exceptions, there are at most two values $ 1 < n<11$ such that $[n]P$ is integral. The exceptions are the points: $P=(6,\pm 18)$ on the curve $E_{108}$, $P=(60,\pm 450)$ on the curve $E_{-13500}$, $P=(4,\pm 12)$ on the curve $E_{80}$, and $P=(84,\pm 756)$ on the curve $E_{-21168}$. In all these exceptions, the point $P$ has exactly 3 integral multiples with $n>1$. There are infinitely many pairs $(E,P)$ of integral points $P$ on the quasi-minimal Mordell curves $E$ such that $[n]P$ is integral for two values of $ 1 < n <5.$ 
\end{theorem}

During the proof, we explicitly determine a construction to find all integral points $P$ on the quasi-minimal Mordell curves such that $[n]P$ is integral for two values of $n<11$.

 Concerning bigger values of $n$, first we will show that if $n \geq 11$ and $[n]P$ is integral, then $n$ is prime. We apply Ingram's techniques \cite{ingram} to prove a gap principle between integral multiples of $P$ and apply our gap principle to prove the following theorem in Section \ref{sec4}.

\begin{theorem} \label{main2}
Let $P$ be non-torsion integral point on the quasi-minimal curve $E_B$. Then $[n]P$ is integral for at most one value $n \geq 29.$
\end{theorem} 
 
In Section \ref{sec5}, based on Theorems \ref{main1} and \ref{main2}, our gap principle and careful consideration of infinite families mentioned in Theorem \ref{main1}, we prove the following result:

\begin{theorem} \label{main 3}
Let $P$ be non-torsion integral point on the quasi-minimal curve $E_B$. Then $[n]P$ is integral for at most three values  $n > 1.$
\end{theorem}

 In \cite{gh2} based on an extensive database, we conjectured that the number of integral points on the quasi-minimal model of  Mordell curves of rank 1 is at most 12. In the last section, we prove some results in this direction for the quasi-minimal Mordell curves $E_B$ with trivial rational torsion subgroups and curves with a rational torsion point of order 2 (All the sixth-power-free, non-square $B$ values).

\section{Preliminaries}
Let $E$ be an elliptic curve defined over $\mathbb{Q}$, $P$ be a point on $E(\mathbb{Q})$, and let $x \left([n]P \right)=\frac{A_n}{D_n^2}$ be in lowest terms  with $D_n >0$. The sequence $\left (D_n^2 \right) $ is called elliptic divisibility sequence defined by $E$ and $P$. We are looking for all the terms $D_k$ such that $D_k=1$. In order to do so, we consider a closely related divisibility sequence.  
Let $K$ be a number field, with ring of integers $R$ and consider an elliptic curve $E$ defined over $K$ by a general Weierstrass equation :
$$ y^2 + a_1xy+a_3y = x^3 +a_2 x^2+a_4x +a_6; \quad a_i \in R,$$
then there exist  polynomials $\phi_{m}, \psi_{m}, \omega_{m} \in \mathbb{Z}\left[a_{1}, \ldots, a_{6}, x, y\right]$ that are defined recursively as in \cite [Exercise III.3.7]{Sil} such that $\phi_m, \psi_m^2$ are relatively prime polynomials in $\mathbb{Z}[x]$. Moreover, Let $P=(x_0,y_0)$ be a point of $E(K)$, $m$ an integer, then $[m] P=\left( \frac{\phi_{m}(x_0)}{\psi_{m}^{2}(x_0)}, \frac{\omega_{m}(x_0,y_0)}{\psi_{m}^{3}(x_0)}\right).$ For more information about division polynomials, their properties and their applications in determining $S$-integral points on elliptic curves  see \cite{ayad}, \cite [Chapter~9] {eke} and \cite{cheon}. 

For Mordell curves $E_B$, these polynomials are defined as follows: 
 \begin{equation} \label{divpol}
\begin{aligned} 
\psi_{1} &=1, \quad \psi_{2}=2 y, \\ 
\psi_{3} &=3 x^{4}+12 B x, \\ 
\psi_{4} &=4 y\left(x^{6}+20 B x^{3}-8 B^{2}\right), \\
 \psi_{2 m+1} &=\psi_{m+2} \psi_{m}^{3}-\psi_{m-1} \psi_{m+1}^{3}, \\ 
 2 y \psi_{2 m} &=\psi_{m}\left(\psi_{m+2} \psi_{m-1}^{2}-\psi_{m-2} \psi_{m+1}^{2}\right), \\
 \phi_1 &=x, \\
 \phi_{m} &=x \psi_{m}^{2}-\psi_{m+1} \psi_{m-1}, \\
  4 y \omega_{m} &=\psi_{m+2} \psi_{m-1}^{2}-\psi_{m-2} \psi_{m+1}^{2} .
 \end{aligned}
\end{equation}
With the notation as above, $\frac{A_n}{D_n^2}=x \left([n]P \right)=\frac{\phi_n(P)}{\psi_n(P)^2}$. $\psi_n$ vanishes at exactly non-trivial $n$-torsion points so $\Div( \psi_n )= \sum_{Q \in E[n]}(Q)- n^2 (O)$. Hence, since $\psi_n^2$ is a polynomial in $x$ with leading coefficient $n^2$ \cite [Exercise III.3.7]{Sil}, we obtain the following key equation:
\begin{equation} \label{hn}
\psi_n(P)^2=n^2 \prod_{Q \in E[n]\backslash \{O\} } \ghl x(P)-x(Q) \ghr.
\end{equation}

If we assign the weights 
$$ w(x)=2, \quad w(4B)=6, $$
then by induction, for odd values of $m$ (respectively even values of $m$), $\psi_m$ (respectively $\psi_m/2y$) is a binary form of degree $m^2-1$ (respectively $m^2-4$) in $x^3$ and $4B$. When $m$ is odd, $\psi_m$ is a polynomial in $\mathbb{Z}$ with leading coefficient $m$, the constant term is zero when $3 \mid m$ and  $\pm(4B)^{\frac{m^2-1}{6}}$ otherwise. $\phi_m $ is always a monic polynomial of degree $m^2$ in $x$. 

The resultant of $\phi_m$ and $ \psi_m^2 $ is $\left(432B^2 \right)^d $, where  
$d=\frac{1}{6}m^2(m^2-1)$\cite[Claim 1 on page 477] {sch}. For a given prime number $p$, let $\ord_p(n)$ denote the $p$-adic valuation of $n$, that is, the exponent of the highest power of $p$ dividing $n$. If $[n]P$ is integral  any prime $p$ that divides $\psi_n (P)$ divides $6B$, and $2 \ord_p(\psi_n(P)) \leq \ord_p(\phi_n(P))$. By Theorem A of \cite{ayad}, these are exactly the primes $p$ that $P \pmod p$ is singular. 

Throughout this paper, we need lower bounds for the canonical height and the difference between the canonical and absolute logarithmic height of the points on the quasi-minimal Mordell curve $E_B$. Fortunately, Voutier and Yabuta \cite{vou} proved some sharp height estimates for Mordell curves. For a rational non-torsion point $P$ on the quasi-minimal Mordell curve $E_B$, Theorem 1.2 of \cite{vou} implies the following result.
\begin{equation} \label{heightkoli}
\hat{h}(P) > \begin{cases}
         \frac{1}{36} \log \ghl B \ghr -0.2247 &\text{if } B < 0, \\
          \frac{1}{36} \log \ghl B \ghr -0.2262 &\text{if } B > 0. \\
     \end{cases}
\end{equation}

In order to fully exploit their result, we use some of the estimates they established for possible values of $B$ modulo powers of 2 and 3 in the course of their proof. These estimates enable us to have stronger bounds for particular values of $B$. The following two lemmas are based on Tables 4 and 5 in \cite{vou}.

\begin{lemma}{\label{height}}
Let $P$ be a rational non-torsion point on the quasi-minimal Mordell curve $E_B$, then $ \hat{h}(P) >  \frac{1}{36} \log \ghl B \ghr -C$
where the value of $C$ is as following 
\begin{itemize}
\item[•] If $B \equiv 15120, 3024 \text{ or } 1296 \pmod{15552}$, then $C=0.2262$,
\item[•] If $B \equiv 80 \text{ or } 208 \pmod{576}$, then $C=0.1347$,
\item[•]If $B \equiv 13392 \text{ or } 9936 \pmod{15552}$, then $C=0.1347$,
\item[•]If $B \equiv 6372, 2052 \text{ or } 324 \pmod{7776}$, then $C=0.1107,$ 
\item[•]If $B \equiv 108 \text{ or } 540 \pmod{3888}$, then $C=0.1107,$
\item[•]If $B  \equiv 1809 \text{ or } 297 \pmod{1944}$, then $C=0.1107,$
\item[•]If $B \equiv 144 \pmod{1728}$, then $C=0.1107,$
\item[•]Otherwise, $C=0.0431.$
\end{itemize}
\end{lemma}

\begin{proof}
The proof is a long but straightforward calculation based on the proof of Theorem 1 and Tables 4 and 5 of \cite{vou}. We present the proof for the case $B \equiv 13392 \text{ or } 9936 \pmod{15552}$, the other cases follow similarly. (See Section \ref{appendix}.)

Let $B \equiv 13392 \text{ or } 9936 \pmod{15552}$ and $P$ be a rational non-torsion point on $E_B$. We denote by $c_p$, the Tamagawa index of $E_B$ at the rational prime $p$, that is the order of the component group, $E_B(\mathbb{Q}_p)/ E_B^0(\mathbb{Q}_p)$, of $E_B$ at $p$, where $E_B^0(\mathbb{Q}_p)$ is the connected component of the identity in $E_B(\mathbb{Q}_p)$. Based on Tables 4 and 5 of \cite{vou}, we have:

\begin{itemize}
 \item[(a)] If $c_p=1$ for all primes, $p>3$, then 
    $$ \hat{h}(P) > \frac{1}{6}\log \Br -1.2921. $$
 \item[(b)] If $c_p \mid 4$ for all primes, $p > 3$ and $2 \mid c_p$ for at least one such prime, then 
     $$ \hat{h}(P) > \frac{1}{24}\log \Br -0.1176. $$
  \item[(c)] If $c_p \mid 3$ for all primes, $p > 3$ and $c_p=3$ for at least one such prime, then 
     $$ \hat{h}(P) > \frac{1}{18}\log \Br -0.4182. $$   
  \item[(d)] If $c_p \mid 12$ for all primes, $p > 3$, $2 \mid c_p$ for at least one such prime, $p$, and $c_q \mid 3$ for at least one other such prime, $q$, then 
   $$ \hat{h}(P) > \frac{1}{36}\log \Br -0.1347. $$  
 \end{itemize}
We compare the lower bounds in each of the above cases with the one mentioned in the lemma. 

For curves in part (a), $\frac{1}{36}\log \Br -0.1347 < \frac{1}{6}\log \Br -1.2921$ unless $\Br \leq 4160$. $B=-2160$ is the only integer that satisfies the stated congruence conditions and $\Br \leq 4160$. The Mordell-weil group of  $E_{-2160}(\mathbb{Q})$ is a torsion-free group of rank 1, with $P=(24,108)$ as a generator. Using PARI/GP \cite{pari}, $\hat{h}(P)>\frac{1}{36} \log(2160)+0.01718 > \frac{1}{36}\log (2160) -0.1347 $. Note that our definition of canonical height is as \cite{Sil} and is half that returned from the height function, ellheight, in PARI. 

The result clearly holds for any rational non-torsion point on the curves in part (b).

\sloppy
Next, we consider the rational points on the curves mentioned in part (c). For all values of $\Br < 27066$, we observe that $\frac{1}{36}\log \Br - 0.1347 < \frac{1}{18}\log \Br - 0.4182$. Thus, it suffices to check rational points on the curves $E_B$, where $B \in \{-21168, -17712, -5616, -2160, 9936, 13392, 25488 \}$. By factoring the $B$-values in this set and referring to Table 1 of [2], we find that the only curve satisfying the conditions of part (c) is $E_{-21168}$. The Mordell-Weil group of $E_{-21168}(\mathbb{Q})$ is a torsion-free group with a rank of 1, and it has $P=(84,756)$ as a generator. Using PARI, we verify that $\hat{h}(P)>\frac{1}{36} \log(21168)-0.1277 > \frac{1}{36}(21168) -0.1347$, as desired.
\end{proof}

When $B$ is a sixth-power-free cube, based on Table 1 in \cite{vou}, the Tamagawa number $c_p$ of $E_B$ at any prime $p > 3$ is either 1 or 3. Therefore, by the same argument as the proof of Lemma \ref{height}, and referring to Tables 4,5 of \cite{vou}, we have the following lemma. (See Section \ref{appendix2}.)
\begin{lemma}{\label{heightcube}}
Let $P$ be a rational non-torsion point on the quasi-minimal Mordell curve $E_B$, where $B$ is a cube, then $ \hat{h}(P) >  \frac{1}{24} \log \ghl B \ghr -C$ where the value of $C$ is as following 
\begin{itemize}
\item[•] If $B$ is odd,  then $C=0.002,$
\item[•] If $B$ is even,  then $C=-0.2290.$
\end{itemize}
\end{lemma}
Moreover, they provide some bounds on the difference between absolute logarithmic and canonical height of an arbitrary non-torsion rational point on the quasi-minimal curve $E_B$:
\begin{equation} \label{height2}
 \begin{aligned}  
   - 0.28 < \frac{1}{2} h(P) - \hat{h} (P) \quad  \text{if  $B < 0$},  \\
  - \frac {1}{6} \log \ghl B \ghr -0.299 < \frac{1}{2} h(P) - \hat{h} (P) \quad  \text{if  $B >0$.} 
  \end{aligned}
\end{equation}

In the course of the proof of Theorem \ref{main1}, we consider the curves with non-trivial rational torsion subgroups separately. It will be particularly beneficial in Section \ref{integral points torsion} when we study the number of integral points on quasi-minimal Mordell curves with a rational torsion point of order 2. It is not difficult to determine the torsion subgroup of all Mordell curves at once. We have the following known lemma to categorize the rational torsion subgroup of Mordell curves. 

\begin{lemma} \label{torsionknapp} \cite [Theorem 5.3 ]{knapp}
Let $E_B$ be a quasi-minimal Mordell curve. Then 

\[   
E_B(\mathbb{Q})_{tors} = 
     \begin{cases}
       \mathbb{Z}/6\mathbb{Z} &\quad\text{if } B=1, \\
       \mathbb{Z}/3\mathbb{Z} &\quad\text{if } B=-432, \\
       \mathbb{Z}/3\mathbb{Z} &\quad\text{if } B \neq 1 \text{ is a square},\\
       \mathbb{Z}/2\mathbb{Z} &\quad\text{if } B \neq 1 \text{ is a cube}, \\ 
       {O}                     &\quad \text{otherwise.}\\
     \end{cases}
\]
\end{lemma}

\begin{remark} \label{torsionknapp2}
Let $B= B_0^2$ be a square different from 1, then $E_B(\mathbb{Q})_{tors}=\{[0,B_0],[0,-B_0],O \}$, and if $B = B_0^3$ is a cube different from 1, then $E_B(\mathbb{Q})_{tors}=\{[B_0,0],O \}.$ 
\end{remark}

\section{Proof of Theorem \ref{main1}} \label{sec3}

This section investigates the integral multiples $[n]P$ of a non-torsion integral point $P$ on the quasi-minimal curve $E_B: y^2=x^3+B$ where $n$ has a divisor less than 11. We will show that if $[n]P$ is integral and $ n >5$, then $n$ has no prime factor less than 11. To achieve this, we consider all the values $n=mc$ with $2 \leq m \leq 10$. We will show that besides a few exceptions, $P$ has at most two integral multiples less than 11, and it happens for infinitely many points. We will explicitly determine all the examples with three integral multiples less than 11. We also consider the curves with non-trivial torsion subgroups separately. 

\subsection{$n=2m$} \label{m=2 section}

In this subsection, we study integral multiples $[2m]P$ of a non-torsion integral point $P$ on the quasi-minimal curve $E_B$. We will show that $[2m]P$ is not integral when $m>10$ or $ 8 \mid m$. In the subsequent subsections, we will prove that $[2m]P$ can be integral only when $m=1$ or 2.
  
There are infinitely many points $P$ on the quasi-minimal curves $E_B$ such that $[2]P$ is integral (see the construction below). Let $P=(a,b)$ be a non-torsion integral point on the Mordell curve $E_B$, so by Remark \ref{torsionknapp}, $ab \neq 0$. Then $x([2]P)=\left(\frac{3a^2}{2b}\right)^2-2a$. Therefore, $[2]P$ is integral if and only if $2b \mid 3a^2$. Hence, if $[2]P$ is integral, $a$ is even, and any prime $p \neq 3$ that divides $b$ must divide $a$ and $B$.

Let $p \neq 3$ be a prime divisor of $b$. Since $B$ is a sixth-power-free integer, from the equation $b^2=a^3+B$, we have $\ord_p(b^2) \neq \ord_p(a^3)$. Therefore, $\ord_p(B)= \min \left (\ord_p(a^3), \ord_p(b^2) \right).$
 
 If $\min \left (\ord_p(a^3), \ord_p(b^2) \right)= \ord_p(a^3)=\ord_p(B),$ then $\ord_p(B)=3$, $\ord_p(a)=1$, and $\ord_p(b) \geq 2$. But, since $\ord_p(2b) \leq \ord_p(a^2)$, this case is not possible when $p=2$. For odd values of $p$ in this case we have $\ord_p(b)=2$, which means $\ord_p(B)=\frac{3}{2} \ord_p(b)$. 
 
 \sloppy
  If $\min \left( \ord_p(a^3), \ord_p(b^2) \right)= \ord_p(b^2)=\ord_p(B),$ then, we have $\left( \ord_p(b), \ord_p(B) \right)= (1,2) \text{ or } (2,4)$. Either way from $3 \ord_p(a) > 2 \ord_p(b)$, we conclude that $\ord_p(a) \geq \ord_p(b)$. 
  
 If $p=3$ divides $b$ and $\ord_p(b) \leq 2\ord_p(a)$, then exactly the same argument works. The only remaining case is when $p=3$ divides $b$, and $\ord_p(b) > 2\ord_p(a)$. A direct 3-adic calculation of the equation $b^2=a^3+B$, shows that it can happen only if $\left( \ord_p(a), \ord_p(b) , \ord_p(B) \right)= \left(0,1,0 \right)$ or $(1,3,3)$. 
 
From the above $p$ adic consideration of divisors of $b$,  we can categorize all possible non-torsion points $P$ on the quasi-minimal Mordell curves $E_B$ with $[2]P$ integral as follows:  
 
%from the equation $a^3+B=b^2$, we have $\ord_p(B) \geq \min \left(\ord_p(a^3), \ord_p(b^2)\right)$
 %Next, let $B$ be a sixth-power-free integer, then from the equation $a^3+B=b^2$, and investigating the possible $p$-adic valuation of all  the primes that divide $a$,$b$ or $B$
\begin{gather} \label{2 barabar}
\begin{cases}
P=(MNt,M^2N), \\
B=M^3N^2K, \\
 M=Nt^3+K    , \\
 x([2]P)=-2tNM+\left( \frac{3t^2N}{2} \right)^2,
\end{cases}
\text{  or  }
\qquad
 \begin{cases}
P=(MNt,3M^2N), \\
B=M^3N^2K, \\
 9M=Nt^3+K    , \\
 x([2]P)=-2tNM+\left( \frac{t^2N}{2} \right)^2,
\end{cases}
\end{gather}

where, $M,N,t$ an $K$ are non-zero integers, $(M,Nt)=1, (K,MNt)=1,$ $Nt$ is even, and $M^3N^2K$ is sixth-power-free.

Considering equation \eqref{2 barabar}, it is easy to construct infinitely many pairs $(E_B,P)$ of quasi-minimal Mordell curves $E_B$ and non-torsion points $P$ on them such that $[2]P$ is integral. The next step is to consider the possibility of such a construction when $E_B$ has a non-trivial torsion subgroup. Let $B$ be a sixth-power-free integer, not equal to 1 or -432 such that $E_B$ has a non-trivial rational torsion subgroup, then from Lemma \ref{torsionknapp}, either $B$ is a cube ($N=1$, $K=k^3$ for some integer $k$), or $B$ is a square  ($M=1$, $K=k^2$ for some integer $k$). By \eqref{2 barabar} in the former case, the question of finding quasi-minimal curves $E_B$ with an integral point divisible by 2 is equivalent to the question of finding all square-free numbers $M$ such that $M$ or $9M$ can be written as sum cubes. However, the latter case clearly contains infinitely many curves with integral points divisible by 2. To see this for any third- power-free $K=k^2$, write $1-k^2$(respectively $9-k^2$) as $t^3N$ where $N$ is third-power-free, Then $P=(Nt,N)$(respectively $P=(N,3N)$) is a point on the curve $E_{N^2k^2}$ where $[2]P$ is integral. 

\begin{lemma} \label{2 barabar height}
 Let $P$ be a non-torsion integral point on the quasi-minimal Mordell curve $E_B$ such that $[2]P$ is integral, then 
$$ \hat{h} (P) <  \frac{7}{18} \log \ghl B \ghr +0.68277. $$  
\end{lemma}

\begin{proof}
Assume $P=(x,y)$ is a non-torsion integral point on $E_B$ in which $[2]P$ is integral. Then by \eqref{2 barabar} $\ghl y \ghr ^3$ divides $3^3 B^2$. Therefore, $ \ghl x^3 \ghr = \ghl B-y^2 \ghr < 9 \ghl  B^{4/3} \ghr + \ghl B \ghr < 10 \ghl B \ghr ^{4/3}$. So if $[2]P$ is integral, we get an upper bound on height of $P$ 
$$ h(P) < \frac{4}{9} \log \ghl B \ghr + \frac{\log 10}{3}, $$ 
and from \eqref{height2}
$$ \hat{h} (P) <  \frac{7}{18} \log \ghl B \ghr +0.68277. $$ 
\end{proof}

\begin{lemma} \label{2m barabar}
Let $P$ be a non-torsion integral point on the quasi-minimal Mordell curve $E_B$, if $[2m]P$ is integral, then $ m<11$.
\end{lemma}
\begin{proof}

Assume $[2m]P$ be integral, and let $Q=[m]P$, from Lemma \ref{height} 
$$ \hat{h} (Q) > m^2( \log \ghl B \ghr / 36 - C) ,$$
Where the value $C$ depends on the congruences condition of $B$ and is given in Lemma \ref{height}.
On the other hand, since $[2]Q=[2m]P$ is integral, $Q$ is integral. Therefore, by Lemma \ref{2 barabar height} 
$$ \hat{h} (Q) <  \frac{7}{18} \log \ghl B \ghr +0.68277. $$
 Comparing these two inequalities, for $m \geq 11$, we obtain:
$$121( \log \ghl B \ghr / 36 - C) <  \frac{7}{18} \log \ghl B \ghr +0.68277. $$
Which means 
\begin{itemize}
\item[•] If $B \equiv 15120, 3024 \text{ or } 1296 \pmod{15552}$, then $ \ghl B \ghr < 12566, $
\item[•] If $B \equiv 80 \text{ or } 208 \pmod{576}$, then $\ghl B \ghr < 303, $
\item[•]If $B \equiv 13392 \text{ or } 9936 \pmod{15552}$, then $\ghl B \ghr < 303,$
\item[•]If $B \equiv 6372, 2052 \text{ or } 324 \pmod{7776}$, then $\ghl B \ghr < 115$ 
\item[•]If $B \equiv 108 \text{ or } 540 \pmod{3888}$, then  $\ghl B \ghr <115 ,$
\item[•]If $B  \equiv 1809 \text{ or } 297 \pmod{1944}$, then  $\ghl B \ghr < 115,$
\item[•]If $B \equiv 144 \pmod{1728}$, then  $\ghl B \ghr < 115,$
\item[•]Otherwise, $\ghl B \ghr < 8.$
\end{itemize}

There are only a few possible $B$ values in any of the above cases. If, for any of this values, $E_B$ has an integral point $P$ with $[2m]P$ integral and $m \geq 11$, Then $E_B$ will have at least 8 non-torsion integral points ($\{ \pm P, \pm [2]P, \pm [m]P, \pm [2m]P$). Moreover, the largest canonical height of integral points on $E_B$ will be at least 484 times of the smallest canonical height of integral points on $E_B$. By applying these two criteria, we can easily check the integral points on the corresponding Mordell curves $E_B$, using the data in \cite{gh2}. A quick search reveals that none of these curves has an integral point $P$ such that $[2m]P$ is integral with $ m \geq 11$. In fact, for integral points $P$ on these curves, there is no integral multiple $[n]P$, with $n> 5$.
\end{proof} 

\subsubsection{$n=4m$}
In this subsection, we will identify all the points $P$ on the quasi-minimal Mordell curves $E_B$ such that $[4]P$ is integral. The first step is to find necessary and sufficient conditions for a point $P$ on the curve $E_B$ so that $[4]P$ is integral. Assume $P=(a,b)$ is a non-torsion integral point on $E_B$, then $[2]P=(a',b')$ where  
\begin{equation} \label{moadele 2 barabar}
a'= \frac{a \left(9a^3-8b^2 \right)}{4b^2},  b'=\frac{-27a^6+36a^3b^2-8b^4}{8b^3} = \frac{a^6+20Ba^3-8B^2}{8b^3}.
\end{equation}
By the arguments in Section \ref{m=2 section}, $[4]P$ is integral if and only if $2b \mid 3a^2$, $a'$ and $b'$ are integers, and $2b' \mid 3a'^2$.

\begin{lemma} \label{no 4 barabar}
Let $P=(a,b)$ be a non-torsion integral point on the quasi-minimal Mordell curve $E_B$. If $\left(\ord_2(a),\ord_2(b)\right)=(1,1)$, then $[4]P$ is not integral.
\end{lemma}
\begin{proof}
Assume $\left(\ord_2(a),\ord_2(b)\right)=(1,1)$, then by equation \ref{moadele 2 barabar}, $a'$ is odd, so $2b' \nmid 3a'^2$. 
\end{proof}
Now we can state the necessary and sufficient conditions on the point $P$ so that $[4]P$ is integral.
\begin{lemma} \label{4 barabar sharte 1}
Let $P=(a,b)$ be a non-torsion integral point on the quasi-minimal Mordell curve $E_B$ and $[2]P=(a',b')$, then $[4]P$ is integral if and only if
\begin{itemize}
 \item[i)]$\left(\ord_2(a),\ord_2(b)\right) \neq(1,1)$,
 \item[ii)] $b'$ is an integer,
  \item[iii)] any prime that divides $b'$ divides $b$.
 \end{itemize}
\end{lemma}
\begin{proof}
Assume $[4]P$ is integral. Then $[2]P$ is an integral point, and by Lemma \ref{no 4 barabar} $\left(\ord_2(a),\ord_2(b)\right) \neq(1,1)$. To prove that prime divisors of $b'$ divides $b$, we examine different prime factors of $b'$.
 
 Let $p \geq 5$ be a prime divisor of $b'$. Since $2b' \mid 3a'^2$, from equation \eqref{moadele 2 barabar},  $-27a^6+36a^3b^2-8b^4 \equiv 0 \pmod p$ and $   a(9a^3-8b^2) \equiv 0 \pmod p$. If $p$ divides $a$, then by the first congruence $p$ divides $b$. So assume  $ 9a^3 \equiv 8b^2 \pmod{p}$. Plugging this congruence in the first congruence, we conclude that $p \mid b$.

Let $p=2$ divides $b'$. Since $a$ is even, if $\ord_2(b)=0$ then by equation \eqref{moadele 2 barabar}, $\ord_2(b')= \ord_2 \left(\frac{-27a^6+36a^3b^2-8b^4}{8b^3}\right)=0$.

Let $p=3$ divides $b'$. If $p$ divides $b'$, then $-27a^6+36a^3b^2-8b^4 \equiv 0 \pmod 3$. Therefore $3 \mid b$.

Conversely assume that $P=(a,b)$, and $[2]P=(a',b')$ are integral points on the quasi-minimal Mordell curve $E_B$, and any prime that divides $b'$ divides $b$. To prove that $[4]P$ is integral it is sufficient to show that for any prime divisor $p$ of $2b$, $\ord_p(2b') \leq \ord_p(3a'^2)$. 

\sloppy
Let $p$ divides b, Since $\ord_p(2b) \leq \ord_p(3a^2)$ $p$ divides $a$, unless $(\ord_3(a),\ord_3(b))=(0,1)$, where in this case $\ord_p(b')=0$. So assume $(\ord_3(a),\ord_3(b))\neq (0,1)$. Since $b^2=a^3+B$, and $B$ is a sixth-power-free integer, $\ord_p(b^2) \neq \ord_p(a^3)$. Therefore, $\ord_p(B)= \min \left( 2\ord_p(b), 3 \ord_p(a)\right).$ 

First, assume that $\min \left( 2\ord_p(b), 3 \ord_p(a)\right)=  2\ord_p(b)= \ord_p(B)$. If $p$ is an odd prime number, then $\ord_p(b')=\ord_p \left(\frac{a^6+20Ba^3-8B^2}{8b^3} \right)= 2\ord_p(B)-3 \ord_p(b)= \ord_p(b).$ On the other hand, $\ord_p (a')= \ord_p \left( \frac{a \left(9a^3-8b^2 \right)}{4b^2} \right)= \ord_p(a)$. Hence, the result follows as desired. If $p=2$, since we assumed $\left(\ord_2(a),\ord_2(b)\right) \neq(1,1)$, we have $\ord_2(a) \geq 2$. With a same argument $\ord_2(b')=\ord_2(b)$, and $\ord_2(a') \geq \ord_2(a)$. The result follows in this case.

Next, assume that  $\min \left( 2\ord_p(b), 3 \ord_p(a)\right)=  3\ord_p(a)= \ord_p(B)$. Since $B$ is sixth-power-free, and $\ord_p(2b) \leq \ord_p(3a^2)$, we have $\ord_p(a)=1, \ord_p(B)=3$ and $p$ is odd. If $p \geq 5$, then $\ord_p(b)=2$ and so $\ord_p(b')=\ord_p \left( \frac{-27a^6+36a^3b^2-8b^4}{8b^3} \right)=0$. The result follows in this case. Finally, if $p=3$, then either $\ord_p(b)=2$, from which we can conclude that $\ord_3(b')=2$ and $\ord_3(a')=1$, or $\ord_3(b)=2$, which in this case $\ord_p(b')=\ord_p(a')=0$.

\end{proof}
 %Assume $ p \geq 5 $ divides $b'$, then  $p$ divides both $-27a^6+36a^3b^2-8b^4$ and $  \left (9a^3-8b^2 \right)$. If $ p \mid a$ then $ p \mid b$, so assume $ 9a^3 \equiv 8b^2 \pmod{p}$  then again from $-27a^6+36a^3b^2-8b^4 \equiv 0 \pmod{p}$ we have $p \mid b$. It is easy to see that if $p=2$ or 3 divides $b'$, then it also divides $b$. So if $4P$ is integral then any prime that divides $b'$ also divides $b$. So the necessary condition is that 
%\begin{equation} \label{4barabar}
%f(a,b)= a^6+20Ba^3-8B^2=  2^{\alpha} \prod_{p \mid b} p^{\gamma_p}. 
%\end{equation}

%If equation holds, then $4P$ is integral if and only if for any prime divisor $p$ of $b$, $\ord_p(2b') \leq \ord_p(3a'^2)$, but if $2P$ is integral, careful consideration of related $p$-adic values shows that this is the always the case unless $\ord_2(a)=\ord_2(b)=1$. Therefore unless $\ord_2(a)=\ord_2(b)=1$, the necessary condition \eqref{4barabar} is sufficient condition for $4P$ to be integral.
 
Let $P=(a,b)$ be a point on the quasi-minimal Mordell curve $E_B$ such that $[4]P$ is integral. We classified all the points $P$ such that $[2]P$ is integral in equation \eqref{2 barabar}. Based on the notation in equation \eqref{2 barabar}, assume that $P=(a,b)=(MNt,3^{\alpha}M^2N)$ ($\alpha=0 \text{ or } 1$). Let $[2]P=(a',b')$. By the argument in the proof of Lemma \ref{4 barabar sharte 1}, $b'=\pm N$ or $\pm 9N$, and $a'=a \cdot l$ for some integer $l$ . As an immediate consequence we prove the following two lemmas.

\begin{lemma} \label{torsion4}
Let $E_B$ be a quasi-minimal Mordell curve with a non-trivial rational torsion subgroup. If $P$ is non-torsion point on $E_B$, then $[4]P$ is not integral. 
\end{lemma} 
\begin{proof}
From Lemma \ref{torsionknapp}, we can assume $B \neq 1, -432$ is a sixth-power-free cube or square. Let $P=(a,b)$ be a point on $E_B$ such that $[4]P$ is integral and $Q=(a',b')=[2]P$. By equation \eqref{2 barabar}, $b=(MNt,3^{\alpha}M^2N)$ ($\alpha=0 \text{ or } 1$), and $B=M^3N^2K$. From the above argument $b'= \pm N$ or $\pm 9N$.  

Let $B=B_1^3$ be a cube. By the above notation, $N=1$. Therefore, $b'= \pm 1 \text{ or } \pm 9$. Hence $a'^3+B_1^3=1$ or $ a'^3+B_1^3=81$. The first equation lead to torsion points $(0, \pm 1)$ on the curve $E_1$, and the second one has no solutions.

If $B=B_1^2$ is a square, then by the above notation, $M=1$. Therefore, $b'= \pm b \text{ or } \pm b/3$. If $b'= \pm b $ then $P$ is a torsion point. The case $b'= \pm b/3$, happens only when $P=(Nt,3N)$. Hence, in this case, we have $Q=(Ntl,N)$ and the following system of equations hold:

\begin{gather} 
\begin{cases*}
 9N^2= N^3t^3+B, \\
  N^2=N^3t^3l^3+B. \\ 
\end{cases*}
\end{gather}
Which means $l=0$ and $P$ is a torsion point.
\end{proof}
With the same argument as the proof of the above lemma, we can show that $[8]P$ is never integral.
\begin{lemma}\label{8 barabar}
For any non-torsion point $P$ on the quasi-minimal curve $E_B$, $[8]P$ is not integral. 
\end{lemma}
\begin{proof}
Let $P$ be an integral point on $E_B$ such that $[8]P$ is integral. With the same notation as \eqref{2 barabar}, $y(P)= 3^{\alpha}M^2N$( where $\alpha = 0 \text{ or } 1$), and $B=M^3N^2K$. Then, for $Q=[2]P$, we have $ \ghl y(Q) \ghr= N \text{ or } 9N$. Since $[4]Q=[8]P$ is an integral point, by the same argument as the proof of Lemma \ref{torsion4}, $ y([2]Q) = \pm y(Q)$. Therefore, $Q$ and so $P$ are torsion points.
\end{proof}
In order to identify all non-torsion points $P$ on the quasi-minimal Mordell curves $E_B$ such that $[4]P$ is integral, we need another version of Lemma \ref{4 barabar sharte 1} that can be checked more easily.
%The conditions that guarantee $4P$ be integral, alongside the notation in \eqref{2 barabar}, enable one to categorize all points $P=(a,b)$ on elliptic curves $E_B$ such that $4P$ is integral. Let
 %$$ a= NMt    \quad b=3^{\alpha} NM^2 \quad B=N^2m^3K, $$ 
\begin{lemma}  \label{4 barabar sharte 2}
Let $P=(a,b)$ be a non-torsion point on the quasi-minimal Mordell curve $E_B$, then $[4]P$ is integral if and only if $2b \mid 3a^2$, $\left(\ord_2(a),\ord_2(b)\right)$ not equal to $(1,1)$, and there exist non-negative integers $\alpha$ and $\gamma_i'$s such that
\begin{equation} \label{4 barabar}
f(a,B)= a^6+20Ba^3-8B^2=  2^{\alpha} \prod_{p_i \mid b} p_i^{\gamma_i}.
\end{equation}
\end{lemma}
\begin{proof}
Condition $2b \mid 3a^2$ is equivalent to condition (ii) in Lemma \eqref{4 barabar sharte 1}. Since $y([2]P)= \frac{a^6+20Ba^3-8B^2}{8b^3}= \frac{-27a^6+36a^3b^2-8b^4}{8b^3}$, equation \eqref{4 barabar} holds if and only if condition (iii) in Lemma \ref{4 barabar sharte 1} holds.
\end{proof}
 
 Let $P=(a,b)$ be a point on the quasi-minimal curve $E_B$, that satisfies the conditions of Lemma \ref{4 barabar sharte 2}. Consider equation \eqref{4 barabar} as a binary form of degree 2 in $a^3$ and $B$. So $\mathcal{F}(X,Y)=X^2+20XY-8Y^2$. Let $\mathcal{A}= \frac{a^3}{\gcd(a^3,B)}$, $\mathcal{B}=\frac{B}{\gcd(a^3,B)}$. Then for any point $P$ such that $[4]P$ is integral we obtain the polynomial $\mathcal{F}\left(\mathcal{A},\mathcal{B} \right)$. With this notation, we have:

\begin{lemma}\label{tabef}
Let $P=(a,b)$ be an integral point on the quasi-minimal Mordell curve $E_B$ such that $[4]P$ is integral, then  $\mathcal{F}(\mathcal{A},\mathcal{B})= \pm 2^3 \cdot 3^{\gamma}$, where $\gamma \in \{0,2,3 \}.$ 
\end{lemma} 

\begin{proof}
Note that if a prime $p \neq 3 $ divides $b$, it divides both $a$ and $B$. 

Let $p \geq 5$ be a prime divisor of $b$. If $\ord_p (a^3) \neq \ord_p(B)$, then exactly one of $\mathcal{A}$ and $\mathcal{B}$ is divisible by $p$ and therefore $p \nmid \mathcal{F}(\mathcal{A},\mathcal{B})$. If $\ord_p(a^3)=\ord_p(B)$, since we assumed $B$ is sixth-power-free $\ord_p(a^3)=\ord_p(B)=3$, from the equation $a^3+B=b^2$, we obtain $\mathcal{A} \equiv -\mathcal{B} \not \equiv 0 \pmod{p}$ . Therefore, $ \mathcal{F}(\mathcal{A},\mathcal{B}) \equiv -27 \mathcal{A}^2 \not\equiv 0 \pmod{p}$.

\sloppy
Let $p=2$. Since $a$ is even, and we ruled out the case $\ord_2(a)=\ord_2(b)=1$, it is easy to see that $\ord_2(a^3) \geq \ord_2(B)+2$. Thus, $\ord_2 (\mathcal{F}(\mathcal{A},\mathcal{B})) =3$.

Let $p=3$. Again, if  $\ord_p (a^3) \neq \ord_p(B)$, then exactly one of $\mathcal{A}$ and $\mathcal{B}$ is divisible by $p$ and therefore $p \nmid \mathcal{F}(\mathcal{A},\mathcal{B})$. So assume $\ord_p(a^3)=\ord_p(B)$. Then we write the equation $\mathcal{F}(\mathcal{A},\mathcal{B})$ as $\left((\mathcal{A}+\mathcal{B})+ 9 \mathcal{B} \right)^2-108 \mathcal{B}^2$. Therefore, $\ord_3 (\mathcal{F}(\mathcal{A},\mathcal{B}))$ can be easily found based on $\ord_3(A+B)$. To be more precise, if $(\ord_3(a),\ord_3(b),\ord_3(B))=(0,0,0)$, then $\gamma =0$ , if $(\ord_3(a),\ord_3(b),\ord_3(B))=(1,2,3)$, then $\gamma =2$, and if $(\ord_3(a),\ord_3(b),\ord_3(B))=(0,1,0)$ or $(\ord_3(a),\ord_3(b),\ord_3(B))=(1,3,3)$, then $\gamma=3$.
\end{proof} 

We rewrite Lemma \ref{tabef} by the notation of \eqref{2 barabar}. Consider the notation \eqref{2 barabar} for a point $P=(a,b)$ such that $[4]P$ is integral to obtain $ \mathcal{A}=Nt^3$, and $\mathcal{B}=K$. Moreover, $Nt^3$ is even, so the equation in Lemma \ref{tabef} became :
$$ \left( \frac{Nt^3}{2}+5K \right)^2 -3 (3K)^2 = \pm 2 \cdot 3^{\gamma}. $$ 
To summarize, we can categorize all points $P$ on the curves $E_B$, where $B$ is sixth-power-free, and $[4]P$ is integral as follows:
\begin{gather} \label{4barbar categorize}
\begin{cases*}
P=(MNt,M^2N), \\
 B=M^3N^2K, \\ 
 M=Nt^3+K    , \\
 \left( \frac{Nt^3}{2}+5K \right)^2 -3 (3K)^2 =  -2, \\
 3 \nmid M.
\end{cases*}
\text{  or  }
\qquad
 \begin{cases*}
P=(MNt,M^2N), \\
 B=(M^3N^2K),\\
 M=Nt^3+K    , \\
\left( \frac{Nt^3}{2}+5K \right)^2 -3 (3K)^2 = \pm 18, \\
3 \mid M.
\end{cases*} 
\\
\begin{align*}
\text{or} 
\begin{cases*}
P=(MNt,3M^2N), \\
B=M^3N^2K \\
 9M=Nt^3+K    , \\
 \left( \frac{Nt^3}{2}+5K \right)^2 -3 (3K)^2 = \pm 54, \\
 3 \nmid N.
\end{cases*}
\end{align*}
\end{gather}
where $(K,MNt)=(M,Nt)=1$, $Nt$ is always even, and $t$ is even unless $\ord_2(x(P))= \ord_2(y(P))=2$. 

As a final remark, assuming $[4]P$ is integral, we can $[2]P$ by the notation in \eqref{2 barabar}, as $y([2]P)=N$, $x([2]P)=Nt$, So $h([2]P) >  \log (B)/3 - \log(2)/3$. On the other hand, $ \ghl x([4]P) \ghr =  \ghl x([2]P) \left( \frac{9t^2}{4}(x([2]P))\pm 2 \right) \ghr $. Since $\ghl x([2]P) \ghr \geq 2 $, we have $\ghl x([4]P)  \ghr > \frac{5}{4} x([2]P)^2 $. Therefore,
\begin{equation} \label{4height}
 h \left( [4]P \right) > \frac{2}{3} \log(B)- 0.239. 
\end{equation}

\subsection{$n=3m$}

In this subsection, we investigate integral multiples $[2m]P$ of a non-torsion integral point $P$ on the quasi-minimal curve $E_B$. We will show that $[3m]P$ is not integral when $m>10$ or when $m$ is even or divisible by 3. In the next section, we will show that $[3m]P$ is not integral when $m=5$ or 7. Therefore, $[3m]P$ can be integral only when $m=1$.
  
Let $P=(a,b)$ be a non-torsion integral point on the quasi-minimal Mordell curve $E_B$. It is clear that $[3]P$ is integral if and only if for any prime that divides $\psi_3(P)$, $2 \ord_p(\psi_3(P)) \leq \ord_p(\phi_3(P))$. Therefore, to identify all points $P=(a,b)$ on the quasi-minimal Mordell curves $E_B$ such that $[3]P$ is integral, we investigate the $p$-adic order of $\psi_3(P)$ and $\phi_3(P)$ for all primes that divide $\psi_3(P)$. The resultant of $\phi_m$ and $ \psi_m^2 $ is $\left(432B^2 \right)^d $ where $d=\frac{1}{6}m^2(m^2-1)$\cite{sch}. Hence, these primes are among the primes that divide $6B$ and are exactly the primes $p$ that $P \pmod p$ is singular \cite[Theorem A]{ayad}. 

Let $P=(a,b)$ be a point on the quasi-minimal Mordell curve $E_B$ such that $[3]P$ is integral. Note that by \eqref{divpol} $$ \psi_3 (P)=3a(a^3+4B), \quad \phi_3(P)=a^9-96Ba^6+48B^2a^3+64B^3. $$

Let $p>3$ be a prime divisor of $B$. From the equation of $\psi_3$, $p$ divides $\psi_3(P)$ if and only if it divides $a$. Let $p>3$ be such a prime with $\ord _p(B)=m$ and $\ord_p (a)=n >0 $. We claim that $ m > n$. On the contrary, assume that  $ m \leq n $, then $\ord_p (\psi_3 (P))=m+n$. On the other hand, $\ord_p( \phi_3(P))=3m$. Thus $\ord_p( \phi_3(P))< 2\ord_p(\psi_3(P))$ and $[3]P$ is not integral. Since $m<6$, considering $ m>n>0$, and equation $b^2=a^3+B$, there are only few possible values for $\ord_p(a)$ and $\ord_p(B)$. With a little bit of work on these possible values, $2 \ord_p (\psi_3(P)) \leq \ord_p( \phi_p(P))$ if and only if  $(\ord_p(a),\ord_p(B))=(1,2),(1,3),(2,4)$. The corresponding $\ord_p \left( x([3]P) \right) $ are respectively as follows: $(\ord_p \left( x([3]P) \right),\ord_p(B))=(0,2),(1,3),(0,4).$

\sloppy
Let $p=2$, again $p$ divides $\psi_3(P)$ if and only if $p$ divides $a$, but unlike the previous case, $p$ does not necessarily divide $B$. By a similar argument as above, let $\ord_2(a)=n$ and $\ord_2(B)=m$, then $ m >n $ unless $m=n=2$ or $(n,m)=(1,0)$. Again by checking the possible values of $\ord_2(a)$ and $\ord_2(B)$, $2 \ord_2 (\psi_3(P)) \leq \ord_2( \phi_p(P))$ if and only if $(\ord_2(a),\ord_2(B))=(2,2),(1,2),(1,3),(2,4),(3,4),(1,0).$ The corresponding $\ord_2 \left( x([3]P) \right) $ is respectively as follows $(\ord_2 \left( x([3]P) \right),\ord_2(B))=(0,2),(1,2),(1,3),(0,4),(0,4),(0,0).$ Note that if $a$ is not even then $x([3]P)$ is not even, and if $x([3]P)$ is even then $\ord_2 \left (y([3]P) \right) \geq 2$.

If we assume $p=3$, then $p$ always divides $\psi_3(P)$. Note that here unlike the previous case, $p$ might divide $\psi_3(P)$ , but does not divide $a$. By the same argument as the cases mentioned above, if $\ord_2(a)=n$ and $\ord_2(B)=m$, then $ m >n $, unless $m=n=0$. By checking the possible values  $\ord_3(a)$, $\ord_3(b)$, and $\ord_3(B)$, $2 \ord_3 (\psi_3(P)) \leq \ord_3( \phi_3(P))$ if and only if   $(\ord_3(a),\ord_3(b), \ord_3(B))=(0, \geq 1, 0),(1,2,3),(1,\geq 3,3).$ The corresponding values for $[3]P=(a',b')$ are as follows: $(\ord_3(a'),\ord_3(b'), \ord_3(B))=(0, \geq 1, 0),(1,\geq 2,3),(1,\geq 3,3).$

So far, we have investigated the prime divisors of $\psi_3(P)$ that divide $6B$. The final step is to prove necessary and sufficient conditions that guarantee no other prime divides $\psi_3(P)$. Note that if a prime p divides a, then it divides $\psi_3(P)$. Therefore, if $[3]P$ is integral, any odd prime $p$ that divides $a$ must divide $B$. Define $\mathcal{F}(P)= \frac{\psi_3(P)}{3a \gcd(a^3,B)}$. Since $a$ has no prime divisor outside the set of divisors of $6B$, $\psi_3(P)$ has no prime divisor that does not divide $6B$, if only if any prime divisor of $\mathcal{F}(P)$ divides $6B$. Hence, $\mathcal{F}(P)=\prod_{p \mid 6B} p^{\gamma_p}$ where $\gamma_p$'s are non-negative integers. With careful consideration of $p$-adic values of $\mathcal{F}(P)$ based on acceptable possible congruences introduced before, we have $\mathcal{F}(P)=3 \cdot 2^{\alpha}$. Where $\alpha \in \{0,3 \}$. The case $\alpha=3$ happens only when $(\ord_2(a),\ord_2(B))=(2,4)$.

We summarize the above argument in the following lemma.

\begin{lemma} \label{3 barabar lem}
All the pairs $(P,E_b)$ of non-torsion points $P$ on the quasi-minimal Mordell curve $E_B$ such that $[3]P$ is integral can be categorized as in table \ref{3 barabar}, where $M$, $N$ and $K$ are relatively prime odd integers, $(3,NK)=1$, and $M^3N^2K$ is a sixth-power-free integer.  
\end{lemma}

\begin{proof}
As we discussed above, $[3]P$ is integral if and only if all prime divisors $p$ of $\psi_3(P)$ divide $6B$, and for these primes, $2\ord_p(\psi_3(P)) \leq \ord_p(\phi_3(P))$. The conditions in columns one and two of Table 1.1 guarantee that any prime divisor of $a$ is a prime divisor of $2B$. This condition is  necessary for $[3]P$ to be integral, as shown above. The numbers represented in the first two columns meet all introduced congruence conditions that warrant $2\ord_p(\psi_3(P)) \leq \ord_3(\phi_3(P))$ for prime divisors of $6B$. By the condition in the third column, there exists an integral point $P$ on the quasi-minimal Mordell curve $E_B$, with given values $x(P)$, and B. As proved above, the condition in the last column means that any prime divisors of $\psi_3(P)$ divide $6B$.
\end{proof}

\begin{table}[tb]
    \centering
    \caption{Integral points $P$ on elliptic curves $E_B$ with $[3]P$ integral.}
    \label{3 barabar}

\begin{center}
\begin{tabular}{ |c|c|c|p{3cm}|p{3cm}| }  
 \hline
       & $x(P)$ & B & Condition for $P$ belongs to $E_B$  & Congruence condition of $\psi_3(P)$      \\
 \hline         
Type I & $2MN$ & $M^3N^2K$ &$M(8N+K)$ is square & $2N+K= \pm 3$ \\
\hline
Type II & $MN$ & $M^3N^2K$ & $M(N+K)$ is square & $N+4K= \pm 3$ \\
\hline
Type III & $4MN$ & $4M^3N^2K$ & $M(16N+K)$ is square & $4N+K= \pm 3$ \\
\hline
Type IV & $2MN$ & $4M^3N^2K$ &$M(2N+K)$ is square & $N+2K= \pm 3$ \\
\hline
Type V & $2MN$ &$ 8M^3N^2K$ &$2M(N+K)$ is square & $N+4K= \pm 3$ \\
\hline
Type VI & $4MN$ &$16M^3N^2K$ &$M(4N+K)$ is square & $N+K= \pm 6$ \\
\hline
Type VII & $8MN$ & $16M^3N^2K$ & $M(32N+K)$ is square & $8N+K= \pm 3$ \\
\hline
\end{tabular}
\end{center}
\end{table}
%Note that the values $M$ and $N$ in the table are odd, and $\ord_3(M) \leq 1 $, $N$ is not divisible by three. 

using the construction in Tabel \ref{3 barabar}, it is easy to find infinitely many pairs $(P,E_B)$ of points $P$ on the quasi-minimal curves $E_B$ such that $[3]P$ integral. As an example, let's introduce infinitely many examples of type II. let $l$ be a non-zero integer. Take $M=1$, $N=12l^2-1$ and $K=1-3l^2$. Note that for infinitely many values of $l$, $M^3N^2K$ is sixth-power-free \cite{erdos}.  Hence, $K,M,N$ satisfy conditions of Lemma \ref{3 barabar lem}. let $P=(12l^2-1,3(12l^2-1)l)$, then $P$ is a non-torsion point on the Mordell curve $E_{(12l^2-1)^2(1-3l^2)}$, where $[3]P$ is integral with $x([3]P)=576l^6 - 288l^4 + 36l^2 - 1$. 

However, if we assume $E_B$ has non-trivial torsion subgroup then it is a different matter.
\begin{lemma} \label{3torsion}
Let $P=(a,b)$ be a non-torsion point on the quasi-minimal Mordell curve $E_B$ with a non-trivial rational torsion subgroup, then $3[P]$ is not integral.
\end{lemma}
\begin{proof}

Since $E_B$ has a non-trivial rational torsion subgroup, by Lemma \ref{torsionknapp}, we can assume $B \neq 1, -432 $ is a cube or square number. 

Let $B$ be a cube. $P$ a point on $E_B$ and $[3]P$ be integral. Considering the notation in Table \ref{4 barabar}, $N=\pm 1$, $K=k^3$, and $P$ is a point of type I, II or V. Based on the condition in the last column of Table \ref{4 barabar}, in each type there will be at most one value for $K$ corresponding to $N= \pm 1$. In view of third column and the possible values of $M$, this process only lead to torsion points.

Let $B$ be a square. $P$ a point on $E_B$ and $[3]P$ be integral. Considering the notation in Table \ref{4 barabar}, $M=\pm 1$, $K=k^2$, and $P$ is a point of type I, II, III, IV, V, or VI. In each type, we found $N$ in terms of $K$ and rewrote the condition in the third column in terms of $K=k^2$ and $M$. This only yields to $P$ be a torsion point.

\end{proof}

\begin{lemma} \label{3m barabar}
There is no integral non-torsion point $P=(a,b)$ on the quasi-minimal Mordell curve $E_B$, where $[3m]P$ is integral, when $m$ is even, divisible by 3 or bigger than 10.
\end{lemma}  
\begin{proof}
Let $P$ be a non-torsion integral point on the quasi-minimal Mordell curve $E_B$ such that $[3m]P$ is integral.

Assume $m$ is even. Then $[2]P$, $[3]P$, and $[6]P$ are all integral. Since $[2]P$ and $[2]([3]P)$ are integral, $a$ and $x([3]P)$ are even. However, we have already seen that when $[3]P$ is integral, $ \ord_2 \left( x([3]P) \right) \leq 1 $ and if  $x([3]P)$ is even, then $\ord_2 \left (y([3]P) \right) \geq 2$. Thus, $ 2 \left( y([3]P) \right) \nmid 3 \left (x([3]P) \right)^2$, and $[6]P$ is not integral.  

Next, assume $m$ is divisible by 3. Then $[3]P$ and $[9]P$ are integral. Let $Q=[3]P=(a',b')$. Since $[3]Q$ is integral, by Lemma \ref{3 barabar lem}, any prime divisor of $a'$, must divide  $2B$. Recall that $a'=\frac{\phi_3(P)}{\psi_3(P)^2}$, and $\phi_3(P)$ is a binary form in $a^3$ and $4B$, with leading term $a^9$, so any prime that divides $2B$ and $a'$ must divide $a$. 

Consider the notation in Table \ref{4 barabar}. Let $a=2^{\alpha}MN$, where $\alpha$ is a non-negaive integer less than 5. According to the argument before Lemma \ref{3 barabar lem}, we obtain $a'=2^{\beta}M$ and $\beta= 0$ or 1. Now if we assume $[3]Q$ is integral, we can locate $Q$ as type I, II, IV or V in Table \ref{4 barabar} with corresponding new value $N= \pm 1$. However, the conditions in third and fourth column of the Table with $N= \pm 1$ only lead to torsion points.  
 
Finally, assume $m >10$. Let $Q$ be a non-torsion point such that $[3]Q$ is integral. By Table \ref{3 barabar}, $x(Q)^2 < 4B$ and so
 $$h(Q) < \frac{\log \ghl B \ghr+ \log(4)}{2}.$$
 Therefore, by inequality  \eqref{height2},
$$ \hat{h}(Q) < \frac{5}{12} \ghl B \ghr +0.646. $$
Since $[3m]P=[3]([m]P)$ is integral, we can take $Q$ as $[m]P$ in above inequality, and $\hat{h}([m]P)$ satisfies the above inequality. On the other hand, if $ m \geq 11$, from Lemma \ref{height} 
$$ \hat{h} ([m]P) >121( \log \ghl B \ghr / 36 - C) ,$$
where the value $C$ depends on the congruences condition of $B$ and is given in Lemma \ref{height}. Comparing these two inequalities, if $[3m]P$ is integral for $m \geq 11$ then 
$$121( \log \ghl B \ghr / 36 - C) <  \frac{5}{12} \log \ghl B \ghr +0.646 .$$
Which means 

\begin{itemize}
\item[•] If $B \equiv 15120, 3024 \text{ or } 1296 \pmod{15552}$, then $ \ghl B \ghr < 13562, $
\item[•] If $B \equiv 80 \text{ or } 208 \pmod{576}$, then $\ghl B \ghr < 317, $
\item[•]If $B \equiv 13392 \text{ or } 9936 \pmod{15552}$, then $\ghl B \ghr < 317,$
\item[•]If $B \equiv 6372, 2052 \text{ or } 324 \pmod{7776}$, then $\ghl B \ghr < 118$ 
\item[•]If $B \equiv 108 \text{ or } 540 \pmod{3888}$, then  $\ghl B \ghr <118 ,$
\item[•]If $B  \equiv 1809 \text{ or } 297 \pmod{1944}$, then  $\ghl B \ghr < 118,$
\item[•]If $B \equiv 144 \pmod{1728}$, then  $\ghl B \ghr < 118,$
\item[•]Otherwise, $\ghl B \ghr < 8.$
\end{itemize}

There are only a few possible values of $B$ in any of the above cases. We have already checked integral points on $E_B$ corresponding to these values in Lemma \ref{2 barabar}.  The integral points $P$ on these curves have no integral multiple $[n]P$, where $n$ is bigger than 5. 
% We check the integral points on the corresponding Mordell curves $E_B$, using the data in \cite{gh2}. None of these curves has an integral point $P$ such that $[3m]P$ is integral with $ m \geq 11$. In fact for integral points on these curve, there is no integral multiple of bigger than 5.
\end{proof}

\subsection{$n=mc, m =5,7$}
The question of finding integral multiples of a point $P$ is indeed a weaker version of the question of the primitive divisor in elliptic divisibility sequence related to the point $P$. So we can use the idea of \cite{ingram2}. For $n \geq 5$, the idea is to embed the points $P$ such that $[n]P$ is integral into the set of solutions of finitely many Thue or Thue-Mahler equations. 
\subsubsection{$m=5$}

Let $P=(a,b)$ and $[5]P$ be integral points on the quasi-minimal Mordell curve $E_B$. By Corollaire A of \cite{ayad}, $\psi_5(a,b) = \pm  \prod_{p \mid 6B}p ^{\gamma_p} $, for some non-negative integers  $\gamma_p$. We have
 $$ \psi_5(x,B)= 5x^{12} + 380Bx^9 - 240B^2x^6 - 1600B^3x^3 - 256B^4. $$
 Consider $\psi_5(x,B)$ as a binary form of degree 4 in $x^3$ and $4B$ with integer coefficients. Let $\mathcal{X}=\frac{x^3}{(4B,x^3)}$, and $ \mathcal{B}= \frac{4B}{(4B,x^3)}$. With this notation we obtain:
\begin{equation} \label{varpsi5}
   \varPsi_5(\mathcal{X},\mathcal{B})= 5\mathcal{X}^{4} + 95 \mathcal{B}\mathcal{X}^3 - 15\mathcal{B}^2\mathcal{X}^2 -25 \mathcal{B}^3\mathcal{X} - \mathcal{B}^4. 
\end{equation}   
   
 \begin{lemma} \label{5psi}
 Let $P=(a,b)$ be a non-torsion point on the quasi-minimal elliptic curve $E_B$ such that $[5]P$ is integral. Let $\mathcal{X}=\frac{a^3}{(a^3,4B)}$ and $\mathcal{B}= \frac{4B}{(a^3,4B)}$. Then 
 $$ \varPsi_5(\mathcal{X},\mathcal{B})= \pm 3^{\alpha} 5^{\gamma},$$ where $\alpha \in \{0,4,6 \} $ and $\gamma \in \{0,1 \}.$
  \end{lemma} 
 \begin{proof}
Since $[5]P$ is integral by \cite[Theorem A]{ayad}, the only primes that might divide $\psi_5(a,B)$ and therefore $\varPsi_5( \mathcal{X}, \mathcal{B} )$ are the prime divisors of $6B$.

Let $p \geq 5 $ divide $B$. If $\ord_p(a^3) \neq \ord_p (4B)$, then in the binary form $\varPsi_5(\mathcal{X},\mathcal{B})$ exactly one of $\mathcal{X}$, or $\mathcal{B}$ is divisible by $p$. Therefore, if $p>5$ or $p=5$ , and $5 \nmid \mathcal{B}$, all the terms beside one term on the right hand side of equation \eqref{varpsi5} is divisible by $p$. Hence $p \nmid \varPsi_5(\mathcal{X},\mathcal{B})$. If $p=5$, and $5 \mid \mathcal{B}$ the term $5\mathcal{X} ^4$ is the term with minimum 5-adic valuation in $\varPsi_5( \mathcal{X}, \mathcal{B})$ , so $\ord_5 (\varPsi_5( \mathcal{X}, \mathcal{B} ))=1$. If $\ord_p(a^3) \neq \ord_p (4B)$, then none of $\mathcal{X}$, or   $ \mathcal{B}$ is divisible by $p$, and from equation $b^2=a^3+B$, we have $\mathcal{X} \equiv - \mathcal{B} \pmod p$. Therefore, $ \varPsi_5( \mathcal{X}, \mathcal{B} ) \equiv 81 \mathcal{X} ^4 \pmod p$, and $p \nmid \varPsi_5( \mathcal{X}, \mathcal{B} )$.  
 
Let $p=2$. If $\ord_2(a^3) \neq \ord_2 (4B)$, exactly one of $\mathcal{X}$ or  $\mathcal{B}$ is odd, so by equation \eqref{varpsi}, $\varPsi_5( \mathcal{X}, \mathcal{B} )$ is odd. If $\ord_2(a^3)= \ord_2(B) $, then both $\mathcal{X}$ and  $\mathcal{B}$ are odd, and therefore $\varPsi_5( \mathcal{X}, \mathcal{B} )$ is odd. 

Let $p=3$. In Lemma  \ref{hnkaranhn}, we will explicitly determine $\ord_3(\psi_n(P))$ for odd values $n$ not divisible by 3. Based on that lemma  $\ord_3 \varPsi_5( \mathcal{X}, \mathcal{B} ) \in \{0,4,6 \}$.
 \end{proof}
From this lemma, it is possible to find all points $P$, where $[5]P$ is integral:
\begin{lemma} \label{5 barabar}
The only non-torsion points $P$ on the quasi-minimal Mordell curves $E_B$ such that $[5]P$ is integral are the points $\pm P$ on the curve $E_{108}$ where $P=(6,18)$. The only integral multiples of $P$ with $n>1$ are $[2]P=(-3,9), [3]P=(-2,-10) \text {  and } [5]P=(366,7002).$
\end{lemma} 
\begin{proof}
Note that since we assumed $B$ is sixth-power-free, any value of $\mathcal{X}$ and  $\mathcal{B}$ give rise to one point on the curve $E_B$. From the previous lemma, the set of points such that $[5]P$ is integral, is embedded into the set of solutions for a finite set of Thue equations. To be more precise the possible pairs $(P,E_B)$ of non-torsion points $P$ on the quasi-minimal Mordell curves $E_B$ such that $[5]P$ is integral is embedded in solutions of the following 12 Thue equations
 $$ 5x^4+95yx^3-15y^2x^2-25yx^3-y^4= \pm 3^{\alpha} 5^{\gamma}, $$
 where $\alpha \in \{0,4,6 \}$, and $\gamma \in \{0,1 \}.$
  
 We solved these Thue equations using PARI and double checked with Magma. The only solutions related to a non-torsion point are the solution $[1,2]$ and $ [-1,-2] $ correspond to the equation $\varPsi_5( \mathcal{X}, \mathcal{B} )= -81$, which gives rise to the points $ (6, \pm 18 )$ on the curve $E_{108}$.   
\end{proof}

\begin{remark}
In \cite{ingram2}, there is a typo in determining $\psi_5(X,Y)$. As a result, the author missed the points $(6, \pm 18)$ on the curve $E_{-81}$ and concluded that the $5^n$ terms in the elliptic divisibility defined by a point $P$ on a quasi-minimal Mordell curve fails to have a primitive divisor. However, this is the only example missed by the author in studying the $5^n$th term of the corresponding elliptic divisibility of points on the quasi-minimal Mordell curves.
\end{remark}

\subsubsection{$m=7$}

 Ingram, \cite{ingram2} proved that there are no elliptic divisibility sequences arising from curves on the quasi-minimal Mordell curves wherein the $7^{\alpha}$th term has no primitive divisor when $\alpha > 0$. Indeed by such a result $[7]P$ is not integral for any point $P$ on the quasi-minimal Mordell curves $E_B$. However, for the sake of completeness and since there are some small typos/unproven statements in his proof, we present the proof with more details. The idea is very similar to the case $m=5$. Let $P=(a,b)$ be a non-torsion point on the quasi-minimal Mordell curve $E_B$ such that $[7]P$ is integral. By Corollaire A of \cite{ayad}, $\psi_7(a,b) = \pm \prod_{p \mid 6B}p ^{\gamma_p} $, for some non-negative integers $\gamma_p$. We have 
\begin{equation*}
\begin{split} 
 \psi_7(x,B)=  & 7x^{24} + 3944Bx^{21} - 42896B^2x^{18} - 829696B^3x^{15} - 928256B^4x^{12} \\
   &  - 1555456B^5x^9 - 2809856B^6x^6 - 802816B^7x^3 + 65536B^8 
 \end{split}
 \end{equation*}
Consider $\psi_7(x,B)$ as a binary form of degree 8 in $x^3$ and $4B$ with integer coefficients. Let $\mathcal{X}=\frac{x^3}{(4B,x^3)}$, and $ \mathcal{B}= \frac{4B}{(4B,x^3)}$. With this notation we obtain
\begin{equation} \label{varpsi}
\begin{split}
   \varPsi_7(\mathcal{X},\mathcal{B})= & 7\mathcal{X}^{8} + 986 \mathcal{B}\mathcal{X}^7 - 2681\mathcal{B}^2\mathcal{X}^6 -12964 \mathcal{B}^3\mathcal{X}^5 - 3626 \mathcal{X}^4 \mathcal{B}^4 \\  & -1519 \mathcal{X}^3 \mathcal{B}^5-686 \mathcal{X}^2 \mathcal{B}^6-49 \mathcal{X} \mathcal{B}^7 + \mathcal{B}^8 . 
   \end{split}
\end{equation}  

\begin{lemma} \label{7psi}
 Let $P=(a,b)$ be a non-torsion point on the quasi-minimal elliptic curve $E_B$ such that $[7]P$ is integral. Let $\mathcal{X}=\frac{a^3}{(a^3,4B)}$ and $\mathcal{B}= \frac{4B}{(a^3,4B)}$. Then 
 $$ \varPsi_7(\mathcal{X},\mathcal{B})= \pm 3^{\alpha} 7^{\gamma},$$ where $\alpha \in \{0,8,12 \} $ and $\gamma \in \{0,1 \}.$
  \end{lemma} 
\begin{proof}
The proof is exactly similar to the proof of Lemma \ref{5psi}. Note that $\gamma =1$ if and only if $7 \mid \mathcal{B}$, and $\alpha >0$ if and only if 3 does not divide any of $\mathcal{X}$ and $\mathcal{B}$, which in this case $ \mathcal{X} \equiv -\mathcal{B} \pmod {3}$.
\end{proof} 
The difference between the case $m=7$, and $m=5$ is that unlike $\varPsi_5(P)$, $\varPsi_7(P)$ is not irreducible over $\mathbb{Q}$. We have
 $$\varPsi_7 (\mathcal{X},\mathcal{B}) = \mathcal{F}_7 (\mathcal{X},\mathcal{B}) \mathcal{G}_7 (\mathcal{X},\mathcal{B}),$$where
$$ \mathcal{F}_7 (\mathcal{X},\mathcal{B})= \mathcal{X}^6+141\mathcal{X}^5 \mathcal{B}^3-363\mathcal{X}^4 \mathcal{B}^4-1924\mathcal{X}^3\mathcal{B}^5-741\mathcal{X}^2 \mathcal{B}^6-48\mathcal{X}\mathcal{B}^5+\mathcal{B}^6,$$ 
and
$$ \mathcal{G}_7 (\mathcal{X},\mathcal{B})= 7 \mathcal{X}^2- \mathcal{X} \mathcal{B} + \mathcal{B}^2 .$$
\begin{lemma} \label{7psi2}
Let $P=(a,b)$ be a non-torsion point on the quasi-minimal elliptic curve $E_B$ such that $[7]P$ is integral. Then 
$$ \mathcal{F}_7(\mathcal{X},\mathcal{B})= \pm 3^{\alpha_1},$$
$$\mathcal{G}_7 (\mathcal{X},\mathcal{B})=  3^{\alpha_2} 7^{\gamma}, $$
where $\alpha_1$ is non-zero if and only $\alpha_2$ is non-zero, $\alpha_1 \in \{ 0,2,3 \}$, $\alpha_1+ \alpha_2 \in \{0,8,12 \}$, and $\gamma \in \{0,1\}$.
\end{lemma} 
\begin{proof} 
The form $7x^2-xy+y^2$ is a positive definite binary form, so $\mathcal{G}_7(\mathcal{X},\mathcal{B})$ is always a positive integer. By Lemma \ref{7psi}, the only primes that might divide $\mathcal{G}_7 (\mathcal{X},\mathcal{B})$ are 3 and 7. Moreover, $\ord_7 \left( \mathcal{F}_7 (\mathcal{X},\mathcal{B}) \cdot  \mathcal{G}_7 (\mathcal{X},\mathcal{B})\right) \leq 1$, and it is 1, if and only if $7 \mid \mathcal{B}$. But if $7 \mid \mathcal{B}$, it is clear that $7 \mid \mathcal{G}_7 (\mathcal{X},\mathcal{B})$. Therefore, $\ord_7 \left( \mathcal{G}_7 (\mathcal{X},\mathcal{B})\right) \leq 1$, and $7 \nmid \mathcal{F}_7 (\mathcal{X},\mathcal{B})$. 

Next, consider the 3-adic valuation of $\mathcal{G}_7 (\mathcal{X},\mathcal{B})$. By the argument in the proof of Lemma \ref{7psi}, we can assume $\mathcal{X} \equiv -\mathcal{B} \not \equiv 0 \pmod 3$, since otherwise $\alpha_1=\alpha_2=0$. Therefore, $\min \left( \ord_3 (\mathcal{X}+\mathcal{B}), \ord_3(2 \mathcal{X}-\mathcal{B})\right)=1$. By the equation
$$\mathcal{G}_7 (\mathcal{X},\mathcal{B})= \left( \mathcal{X}+ \mathcal{B} \right)^2 +3 \mathcal{X} \left( 2\mathcal{X}- \mathcal {B}\right),$$
 we obtain $\ord_3 \left( \mathcal{G}_7 (\mathcal{X},\mathcal{B})\right)\geq 2$. If $\ord_3 (\mathcal{X}+\mathcal{B}) >1 $ or $\ord_3(2 \mathcal{X}-\mathcal{B})>1$, then $\ord_3 \left( \mathcal{G}_7 (\mathcal{X},\mathcal{B})\right)=2$. 
If $\ord_3 (\mathcal{X}+\mathcal{B})= \ord_3(2 \mathcal{X}-\mathcal{B})=1$, then by equation
$$ \mathcal{G}_7 (\mathcal{X},\mathcal{B})-3 (\mathcal{X}+\mathcal{B}) (2 \mathcal{X}-\mathcal{B})= \left(\mathcal{X}-2 \mathcal{B} \right)^2,$$
we have $\ord_3 \left( \mathcal{G}_7 (\mathcal{X},\mathcal{B})\right)=2, \text{ or } 3$. The last statement of the lemma follows from Lemma \ref{7psi}.
\end{proof}

\begin{lemma} \label{$n=7$}
Let $P=(a,b)$ be a rational non-torsion point on the quasi-minimal Mordell curve $E_B$. Then $[7]P$ is not integral.
\end{lemma}

\begin{proof}
From Lemmas \ref{7psi} and \ref{7psi2}, the pairs $(P,E)$ of non-torsion points $P$ on the quasi-minimal Mordell curves $E_B$ such that $[7]P$ is integral are embedded in the set of non-trivial primitive solutions of the following system of equation
\begin{equation*}
\begin{cases}
x^6+141x^5y-363x^4y^2-1924x^3y^3-741x^2y^4-48xy^5+y^6= \pm 3^{\alpha_1} \\
7x^2 -xy + y^2= 3^{\alpha_2} 7^{\gamma},
\end{cases}
\end{equation*}
where $(\alpha_1,\alpha_2) \in \{(0,0),(5,3),(6,2),(9,3),(10,2) \}$, and $\gamma=0 \text { or } 1$.

The only non-trivial primitive solutions are $[1,-1]$ and $[-1,1]$, which correspond to the system defined by $(\alpha_1,\alpha_2,\gamma)=(6,2,0)$ with a positive sign in the first equation, and the solutions $[2,1]$, $[-2,-1]$, $[1,-4]$ and $[-1,4]$ correspond to the system defined by $(\alpha_1,\alpha_2,\gamma)=(9,3,0)$ with a negative sign in the first equation. None of these solutions lead to a non-torsion point on the quasi-minimal curve $E_B$. 

In order to find the solution to the above system of equations, we can apply two different approaches. The first one is to consider the first equation as a Thue-equation to determine the primitive solutions using Magma or PARI and then check the solutions in the second equation. The other approach is to solve the second equation of each system by using the algorithm described in \cite [p. 75]{dickson} and then check the solutions in the first equation. 
\end{proof}
\subsection{Points with more than two integral multiples $[n]P$, $2 \leq n \leq 10$}

So far, we have proved there is no point $P$ on a quasi-minimal Mordell curve $E_B$ with integral multiple $[n]P$, $ 6 \leq n \leq 10$. When $n=5$, there are only two points $P$ with $[5]P$ integral. To complete the proof of Theorem \ref{main1}, we proceed by investigating the points $P$ for which $[2]P$ and $[3]P$ are integral. Subsequently, we examine weather $[4]P$ is also integral. This approach allows us to identify all the possible points $P$ on the quasi-minimal Mordell curves with more than two integral multiples $[n]P$, where $2 \leq n \leq 10$, in addition to the aforementioned ones.
 To begin, we determine all points that satisfy equations \eqref{2 barabar} and the conditions outlined in Table \ref{3 barabar} simultaneously. In Proposition \ref{2,3}, we will see that there are infinitely many pairs $(P,E_B)$ of points $P$ on the quasi-minimal Mordell curve $E_B$ that satisfy both set of conditions. We can categorize all these pairs and use equation \eqref{4barbar categorize} to examine if, for any one of them, $[4]P$ is integral. The following proposition summarizes this process. Note that the first two families correspond to the value $t=2$ in equations \eqref{2 barabar}. 

\begin{proposition} \label{2,3}
Let $M$,$N$, and $K$ be non-zero pairwise relatively prime integers such that $M$ and $K$ are odd, $3 \nmid NK$, and $M^3N^3K$ is sixth-power-free, then all non-torsion points $P$ on the quasi-minimal curves $E_B$ with $[2]P$ and $[3]P$ integral at the same time can be categorized as following:

\begin{itemize}  

\item[•] The first family ($t=2$, and $ 3 \mid M$)
\begin{gather*}
\begin{cases}
P=(2MN,\pm M^2N), \\
 B=M^3N^2K, \\ 
2N+K= \pm 3   , \\
 8N+K=M. \\
\end{cases}
\Rightarrow
\begin{cases}
x=6(2N+1)N, \\
y=\pm 9(2N+1)^2N, \\
B=27(2N+1)^3N^2(-2N+3), \\
x ([3]P)=16N^4-16N^3-12N^2+2N+1. \\
\end{cases}
\end{gather*}
 Considering equation \eqref{4barbar categorize}, the only point in this family with $[4]P$ integral  corresponds to $N=2$. This yields to the points, $\pm P$ on the curve $E_{-13500}$ where $P=(60, 450)$. Note that the only integral multiples of these points are $\pm P$, $\pm [2]P$, $\pm [3]P$, and $\pm [4]P$. In these equations, $N \equiv 2 \pmod{3}$ and if $N \neq 2$, $h([3]P) > 0.57 \log (B) $. Moreover, $x([3]P)$ is always greater than $2 \ghl B \ghr ^{1/3}.$
\item[•] The second family ($t=2$, and $\ord_3(y(P))> \ord_3(x^2)$)
\begin{gather*}
\begin{cases}
P=(2MN,\pm 3M^2N), \\
 B=M^3N^2K, \\ 
2N+K= \pm 3, \\
 8N+K=9M. \\
\end{cases}
\xRightarrow[]{N=3c+1}
\begin{cases}
x=12c^2+10c+2, \\
y=\pm 36c^3+48c^2+21c+3, \\
B=-432c^6-864c^5-648c^4-208c^3 \\ 
 \quad \quad  -15c^2+6c+1, \\
x ([3]P)= 144c^4 + 144c^3 + 36c^2 \\ 
 \quad \quad \quad \quad - 2c - 1. 
\end{cases}
\end{gather*}
By checking  equation \eqref{4barbar categorize}, we obtain no non-torsion integral point $P$ where $[4]P$ is integral in this family. For the points $P$ in this family, we have $h([3]P) > 0.66 \log (B) $,  $x([3]P)$ is always greater than $2 \ghl B \ghr ^{1/3} $.  

\item[•] The third family ($t=1$, $\ord_2(N)=1$, and $\ord_3(M)=1$) 
\begin{gather*}
\begin{cases}
P=(MN,M^2N), \\
 B=M^3N^2K, \\ 
N+4K= \pm 6   , \\
 N+K=M. \\
\end{cases}
\Rightarrow
\begin{cases}
x=(6-3K)(6-4K), \\
y=\pm (6-3K)^2(6-4K), \\
B=K(6-3K)^3(6-4K)^2, \\
x ([3]P)= 16K^4 - 80K^3 + 132K^2 - 74K + 4.\\
\end{cases}
\end{gather*}
By the conditions on 2-adic valuation of $x(P)$ and $B$, $[4]P$ cannot be integral.  In these equations $K \equiv 1 \pmod{6}$,  and if $N \neq 1$, $h([3]P) > 0.59 \log (B) $. Moreover, if $N \neq 1$ then $x([3]P)$ is  greater than $2 \ghl B \ghr ^{1/3}$.  Note that  $N=1$ corresponds to $E_{108}$. 

\item[•] The fourth family ($t=1$, $\ord_3(y(P))> \ord_3(x^2)$ and $\ord_2(N)=1$)
\begin{gather*}
\begin{cases}
P=(MN,3M^2N), \\
 B=M^3N^2K, \\ 
N+4K= \pm 6   , \\
 N+K=9M. \\
\end{cases}
\Rightarrow
\begin{cases}
x=(12M-2)M, \\
y=\pm 3 M^2 (12M-2), \\
B= M^3(12M-2)^2(2-3M), \\
x ([3]P)=144M^4 - 144M^3 + 36M^2 - 2M,  \\
\end{cases}
\end{gather*}
By the condition of 2-adic valuation of $x(P)$ and $B$, $[4]P$ can not be integral. For the points $P$ in this family, we have $h([3]P) > 0.66 \log (B) $,  $x([3]P)$ is always greater than $2 \ghl B \ghr ^{1/3} $.  

\item[•] The fifth family($t=1$,$\ord_3(M)=1$, and $\ord_2(N)=2$) 
\begin{gather*}
\begin{cases}
P=(MN,M^2N), \\
 B=M^3N^2K \\ 
N+4K= \pm 24   , \\
 N+K=M \\
\end{cases}
\Rightarrow
\begin{cases}
x=(24-3K)(24-4K), \\
y=\pm (24-3K)^2(24-4K), \\
B=k(24-3K)^3(24-4K)^2, \\
x ([3]P)= K^4 - 20K^3 + 132K^2 - 296K + 64.\\
\end{cases}
\end{gather*}
By checking  equation \eqref{4barbar categorize}, we obtain no non-torsion integral point $P$ where $[4]P$ is integral on this family. We might take $K \equiv 1 \pmod{6}$. When $K=1$, it give rise to the point $\pm P$ where $P=(420,8820)$ on the curve $E_{370440}$ where the only integral multiples are $ \pm [2]P$ , $ \pm [3]P$. If we take $K=7$, it gives rise to the point $\pm P$ where $P=(-12,36)$ on the curve $E_{3024}$ where the only integral multiples are $ \pm [2]P$ , $ \pm [3]P$. When $K \neq 1,7$. $x([3]P) > 2 \ghl B \ghr ^{1/3}$ and $h([3]P)> 0.44 \log \ghl B \ghr $.

\item[•] The sixth family($t=1$, $\ord_3(y(P))> \ord_3(x^2)$ and $\ord_2(N)=2$)
\begin{gather*}
\begin{cases}
P=(2MN,3M^2N), \\
 B=M^3N^2K \\ 
N+4K= \pm 24   , \\
 N+K=9M \\
\end{cases}
\xRightarrow[]{K=6c+5}
\begin{cases}
x=48c^2 - 32c + 4 \\
y=\pm  -288c^3 + 336c^2 - 120c + 12 \\
B= -27648c^6 + 27648c^5 + 6912c^4 \\
 \quad \quad - 17920c^3 + 7872c^2 - 1344c + 80 \\
x ([3]P)= 144c^4 - 72c^2 + 16c + 1 \\
\end{cases}
\end{gather*}
By checking equation \eqref{4barbar categorize}, the only points in this family with $[4]P$ integral  corresponds to $c=0$ and $c=-1$. This yields to points,$\pm P$, where $P=(4, 12)$ on the curve $E_{80}$ and $Q=(84,756)$, on the curve $E_{-21168}$. Note that the only integral multiples of these points are $\{ \pm P, \pm [2]P, \pm [3]P, \pm [4]P \}$ and $\{ \pm Q, \pm [2]Q, \pm [3]Q, \pm [4]Q \}$. If $C \neq 0,1 $ then $x([3]P) > 2 \ghl B \ghr ^{1/3}$, and $h([3]P)> 0.51 \log \ghl B \ghr $.
\end{itemize}
\end{proposition}

\begin{remark} \label{2,3 height}
Based on the above calculations, Let $P$ be a non-torsion integral point on the curve $E_B$ such that $[2]P$ and $[3]P$  are integral. If $[n]P$ is integral for some $n >10$, then $x([3]P) > 2 \ghl B \ghr ^{1/3}$ and $h([3]P)> 0.44 \log \ghl B \ghr $.
\end{remark}

\begin{corollary} \label{binahayat 2,3}
There are infinitely many pairs $(P,E_B)$ of non-torsion points $P$ on the quasi-minimal Mordell curves $E_B$ such that $P$ has at least two integral multiples $[n]P$, with $n>1$.
\end{corollary}
\begin{proof}
By a theorem of Erdos\cite{erdos}, each one of the families introduce in Proposition \ref{2,3}, contains infinitely many sixth-power-free values $B$.
\end{proof}

\begin{corollary} \label{2,3,4}
The points $P=(60,\pm 450)$ on the curve $E_{-13500}$, $P=(4,\pm 12)$ on the curve $E_{80}$, and $P=(84,\pm 756)$ on the curve $E_{-21168}$ are the only non-torsion points on the quasi-minimal Mordell curves $E_B$ that have more than 2 integral multiples $[n]P$, with $1<n <5$.
\end{corollary}

\begin{remark} \label{nokte10}
All examples of points $P$ mentioned in Theorem \ref{main1}, with three integral multiples $[n]P$, $n <11$, have exactly three integral multiples. 
\end{remark}

Theorem \ref{main1} follows from Lemmas \ref{2m barabar}, \ref{8 barabar}, \ref{3m barabar}, \ref{5 barabar}, \ref{$n=7$}, Corollaries \ref{binahayat 2,3}, and \ref{2,3,4}, and Remark \ref{nokte10}.

\section{Proof of Theorem \ref{main2} } \label{sec4}
This section follows Ingram's techniques \cite{ingram} to investigate integral multiples $[n]P$ when $n >10$. First, we consider the case where $ \Br  \leq 75$. Among the 11 curves in this range with more than four non-torsion integral points, none of them possesses an integral point $P$ where $[n]P$ is integral for $n > 5$. So for the entirety of this section, we assume $B>75$. From Theorem \ref{main1}, any such integer $n$ has no prime factor less than 11. Therefore, throughout this section, we assume $n$ has no prime factor less than 11. Under the assumption $[n]P$ is integral, we use equation \eqref{hn} to find an upper bound on the height of $P$.
\begin{lemma} \label{hnkaranhn}
Let $n>10$, and $P=(a,b)$ be a point on the quasi-minimal Mordell curve $E_B$ such that $[n]P$ is integral, then $\ghl \psi_n(P) \ghr < \left(3^{3/2}2^2 \Br \right)^{\frac{n^2-1}{6}}$.
\end{lemma}
\begin{proof}

By Theorem \ref{main1}, $n$ is odd and not divisible by 3. Therefore, by the argument following equation \ref{hn}, for any point $P=(a,b)$ on the curve $E_B$, $\psi_n(P)$ is a binary form of degree $\frac{n^2-1}{6}$ in terms of $a^3$ and $4B$. The coefficient of the term $(a^3)^{\frac{n^2-1}{6}}$ in $\psi_n(P)$ is $n$, and the coefficient of the term $(4B)^{\frac{n^2-1}{6}}$ is $\pm1$. Recall that
 $$x \left([n]P \right)=\frac{\phi_n(P)}{\psi_n(P)}=\frac{\phi_n(P)}{\psi_n(P)^2}.$$ Since $x \left( [n]P \right)$ is integral, $\gcd(\psi_n(P),\phi_n(P))=\psi_n(P)$. As we have seen before, it means that $\psi_n(P)=\pm  \prod_{p \mid 6B} p^{\gamma_p}$. for some non-negative integers $\gamma_p$'s. To bound $\ghl \psi_n(P) \ghr$, we find upper bounds on the $p$-adic valuations of $\psi_n(P)$ for different prime divisors $p$ of $6B$. 

Let $p >3 $ be a prime divisor of $B$ and $\psi_n(P)$. Since $p \mid \phi_n(P)$, and $\phi_n(P)$ is monic in $a^3$; $p \mid a.$ With the assumption $B$ being sixth-power-free, from the equation $b^2=a^3+B$, we have $\ord_p(B)\leq \ord_p(a^3)$. First Assume $ \ord_p(B)< \ord_p(a^3)$, considering $\psi_n(P)$ as a binary form in $a^3$ and $4B$, the term with minimum $p$-adic valuation is $(\pm 4B)^{\frac{n^2-1}{6}}$. Hence, in this case, $\ord_p(\psi_n(P))=\frac{n^2-1}{6} \ord_p(B)$. Next, assume $\ord_p(a^3)= \ord_p(B)=3$. Let $a_1=\frac{a^3}{p^3}$ and $B_1=\frac{B}{p^3}$. Since $\ord_p(a^3+B)$ is even, $a_1 \equiv -B_1 \pmod{p}$. Therefore, $\frac{\psi_n(P)}{p^{3 \frac{n^2-1}{6}}} \equiv a_1^{\frac{n^2-1}{6}} \psi_n(1,-1) \pmod p$. But, $\psi_n(1,-1)$ is a power of 3 \cite{ingram2}. Hence, $\ord_p(\psi_n(P))=\frac{n^2-1}{6} \ord_p(B)$.

To evaluate $\ord_2(\psi_n(P))$,again consider $\psi_n(P)$ as a binary form in $4B$ and $a^3$ of degree $\frac{n^2-1}{6}$. The leading term of  $\psi_n(P)$  in terms of $a^3$ is $n$, and it is 1 in terms of $4B$. If $\ord_2(a^3) \neq \ord_2 (4B)$, then $\ord_2(\psi_n(P))= \frac{n^2-1}{6} \left( \min (\ord_2(a^3) , \ord_2(4B) \right)$. If $\ord_2(a^3) = \ord_2 (4B)=6$, then by induction and recursive formula in \eqref{divpol} we have
\[
  \ord_2(\psi_n(P))= 
  \begin{cases}
   n^2-1 \quad \text{ if $ 3 \not \mid n$,} \\
   \geq n^2     \quad  \text{if $3 \mid n$.}                                  
  \end{cases}
\]

Finally, we find $\ord_3(\psi_n(P))$. If 3 divides only one of $a$ or $B$ or if 3 divides both of them, but $\ord_3(a^3) \neq \ord_3(B)$, then since $ 3 \nmid n$, and $\psi_n(a^3,B)$ is a binary form in $a^3$ and $B$. By the same argument as above, $\ord_3(\psi_n(P))= \frac{n^2-1}{6} \left( \min (\ord_3(a^3) , \ord_3(B) \right)$. For the remaining cases, we use simple induction and recursive relations in \eqref{divpol} to calculate the 3-adic valuation of $\psi_n(P)$.
\begin{itemize}
\item[•] If $3 \nmid abB$
\[
  \ord_3(\psi_n(P))= 
  \begin{cases}
  \geq 0 \quad \text{ if $ 3 \mid n$,} \\
   0     \quad  \text{otherwise.}                                  
  \end{cases}
\]

\item[•] If $3 \nmid aB$ and $3 \mid b$ 
\[
  \ord_3(\psi_n(P))= 
  \begin{cases}
  \frac{n^2-1}{4} \quad \text{ if $n$ is odd,} \\
  \geq \frac{n^2}{4}    \quad  \text{if $n$ is even.}                                  
  \end{cases}
\]

\item[•] If $\ord_3(a^3)=\ord_3(B)=3$ , and $\ord_3(b)= 2$
\[
  \ord_3(\psi_n(P))= 
  \begin{cases}
  \geq \frac{2}{3} n^2 \quad \text{ if $ 3 \mid n$,} \\
   \frac{2}{3} (n^2-1)     \quad  \text{if $ 3 \nmid n$.}                                  
  \end{cases}
\]

\item[•] If $\ord_3(a^3)=\ord_3(B)=3$ , and $\ord_3(b)> 2$
\[
  \ord_3(\psi_n(P))= 
  \begin{cases}
  \geq \frac{3}{4} n^2 \quad \text{ if $ n$ is even,} \\
   \frac{3}{4} (n^2-1)     \quad  \text{if $n$ is odd.}                                  
  \end{cases}
\]
\end{itemize}
 
 This completes the proof.
 
\end{proof}

The next step to fully exploit equation \eqref{hn} is to find an upper bound on the $\ghl x(Q) \ghr$ when $Q$ is a non-identity torsion point of order dividing $n$ on $E(\mathbb{C})$. David \cite{david} determined a bound in the general case. But for a better result, we use the same idea as the one in \cite{ingram} to bound $\ghl x(Q)\ghr $.

\begin{lemma} \label{Qnkaran}
Let $Q$ be a non-identity, torsion point of order dividing $n\geq 11$ on $E_B(\mathbb{C})$, then $\ghl x_Q \ghr < \left( \frac{n^2}{7} \right) \Br^{\frac{1}{3}}$.
\end{lemma}
\begin{proof}
We know that all curves $E_B$ are isomorphic over $\mathbb{C}$. Recall the isomorphism 
\[
 \begin{aligned} E_{B}(\mathbb{C}) & \simeq E_{1}(\mathbb{C}) \\(x, y) & \mapsto\left(x B^{-1/3}, y B^{-1 / 2}\right). \end{aligned}
\]
Therefore, it is sufficient to prove the lemma for $E_1$. $E_1$ has complex multiplication; its corresponding lattice is $\Lambda= \omega \mathbb{Z}[\rho]$, where $4.2065 < \omega <4.2066 $ is the real period of $E_1$, and $\rho= e^{2 \pi i /3}$. Let $\Lambda \subset \mathbb{C}$ be a lattice, the related Weierstrass $\wp$-function is defined by 
$$ \wp (z;\Lambda)=\frac{1}{z^2}+ \sum_{\substack{ \omega \in \Lambda  \\ \omega \neq 0}}  \left( \frac{1}{(z-\omega)^2}-\frac{1}{\omega ^2} \right), $$
and there is an isomorphism 

\[
 \begin{aligned} \mathbb{C} / \Lambda &\rightarrow E_{1}(\mathbb{C}) \\ z & \mapsto\left( \wp(z;\Lambda),     \frac{1}{2} \wp'(z; \Lambda)  \right). \end{aligned}
\]
 
 Choose a representative for $z$  in $\mathbb{C} / \Lambda$ such that $|\operatorname{Re}(z)|,|\operatorname{Im}(z)| \leqslant \omega_{1} / 2$. Then $|u-z| \geqslant|u| / 2$ for all $u \in \Lambda$, and therefore,
$$
\left|\sum_{\substack{u \in \Lambda \\ u \neq 0}}\left(\frac{1}{(u-z)^{2}}-\frac{1}{u^{2}}\right)\right| \leqslant 2|z| \sum_{\substack{u \in \Lambda \\ u \neq 0}} \frac{4}{|u|^{3}}+|z|^{2} \sum_{\substack{u \in \Lambda \\ u \neq 0}} \frac{4}{|u|^{4}}.
$$
We recall another known lemma. Let $\Lambda$ be a lattice, $\sigma$ be the minimum distance between points in $\Lambda$  and $c= 8 \pi/ \sigma^2$, then for any $k>2$ 
 $$ \sum_{\substack{ \omega \in \Lambda  \\ \ghl \omega \ghr \geq  1}} \frac{1}{\ghl \omega \ghr^k } \leq c   \zeta (k-1).$$
Therefore, 
$$ \ghl \wp  (z;\Lambda)\ghr  <  \ghl z \ghr ^{-2} + 8 \ghl z \ghr \frac{2 \pi}{\omega^2} \zeta (2) + 4 \ghl z \ghr ^2 \frac{2 \pi}{\omega ^2} \zeta(3). $$

Since $\ghl z \ghr < \frac{\omega}{\sqrt{2}},$ and $  4.2065 < \omega < 4.2066 $  we obtain

$$ \ghl \wp (z;\Lambda) \ghr < z^{-2} +9.2325. $$
if $z$ is a non-zero torsion point of order dividing $n$, then $\ghl z \ghr  > \frac{\omega}{n}$. Hence, 
$$ \ghl \wp(z;\lambda) \ghr < \frac{n^2}{\omega ^2} +9.2325. $$
Since we assumed $n \geq 11$ , we have 
$$ \ghl \wp(z;\Lambda) \ghr < \frac{n^2}{7}, $$
and the result follows from the first isomorphism mentioned in the proof.

\end{proof}

Lemmas \ref{Qnkaran} and \ref{hnkaranhn} enable us to bound  the height of integral points $P$ in which $[n]P$ is integral.

\begin{lemma} \label{heightkaran}
Let $P$ be a non-torsion rational point on the quasi-minimal Mordell curve $E_B$ such that $[n]P$ is integral for some integer $n \geq 11$, then 
$$\hat{h} (P) < \log n + \frac{\log B}{6} -0.617  \quad \text{ if $B<0$ }, $$
$$ \hat{h} (P) < \log n + \frac{\log B}{3} -0.597  \quad \text{ if $B > 0$ }. $$
\end{lemma}  

\begin{proof}
Let $[n]P$ be integral. By equation \eqref{hn} and Lemma \ref{hnkaranhn}, we have:
$$ \min_{\substack{[n]Q =O  \\ Q \neq O}} \ghl x(P)- x(Q) \ghr ^{n^2-1} < \ghl 3^{3/2}.4.B \ghr ^{\frac{n^2-1}{3}}.$$
From the bound on $x(Q)$ in Lemma \ref{Qnkaran} and the assumption $n \geq 11$, we obtain:
$$ \ghl x(P) \ghr < \frac{n^2}{6} \Br^{1/3} \Rightarrow h(P) < 2 \log n + \frac{1}{3} \log \Br - \log 6. $$
The result follows from equation \eqref{height2}.
\end{proof}

This upper bound will be crucial in our next steps. As the first application, we prove the following lemma:

\begin{lemma} \label{prime}
Let $P$ be a non-torsion point on the quasi-minimal Mordell curve $E_B$ such that $[n]P$ is integral for some integer $n >10$. Then $n$ is prime.
\end{lemma}
\begin{proof}
Assume to the contrary that $n$ is composite. Let $n=q_1.q_2$ where $q_1 \leq q_2$. Then $[q_1]([q_2]P)$ is integral, so by Lemma \ref{heightkaran}, we have: 
$$ \hat{h} ([q_2]P) < \log (q_1)+  \frac{\log B}{3} -0.597 .$$
On the other hand, $\hat{h} ([q_2]P) > q_2^2 \hat{h}(P)> q_1^2 \hat{h}(P), $
so we have 
$$ q_1^2 \left( \frac{\log B}{36} -C \right) < \log (q_1)+  \frac{\log B}{3} -0.597 ,$$
where the value $C$ depends on the congruence condition of $B$ and is given in Lemma \ref{height}.
By Theorem \ref{main1}, $q_1 \geq 11$, this means: 
\begin{itemize}
\item[•] If $B \equiv 15120, 3024 \text{ or } 1296 \pmod{15552}$, then $ \ghl B \ghr < 15283, $
\item[•] If $B \equiv 80 \text{ or } 208 \pmod{576}$, then $\ghl B \ghr < 395, $
\item[•]If $B \equiv 13392 \text{ or } 9936 \pmod{15552}$, then $\ghl B \ghr < 395,$
\item[•]If $B \equiv 6372, 2052 \text{ or } 324 \pmod{7776}$, then $\ghl B \ghr < 152$ 
\item[•]If $B \equiv 108 \text{ or } 540 \pmod{3888}$, then  $\ghl B \ghr < 152,$
\item[•]If $B  \equiv 1809 \text{ or } 297 \pmod{1944}$, then  $\ghl B \ghr < 152,$
\item[•]If $B \equiv 144 \pmod{1728}$, then  $\ghl B \ghr < 152,$
\item[•]Otherwise, $\ghl B \ghr < 11.$
\end{itemize}

There are only a few possible values of $B$ in any of the above cases. We have already checked integral points on $E_B$ corresponding to most of these values in Lemma \ref{2 barabar}. The integral points $P$ on these curves have no integral multiple $[n]P$, where $n$ is bigger than 5. For the few remaining values, we applied the same method used in Lemma \ref{2 barabar} to investigate integral multiples of ration points on $E_B$. Similarly, for the integral points $P$ on these curves, we found no integral multiple $[n]P$, where $n$ is greater than 5.

%There are only a few possible $B$ values in any of the above cases.  We checked the corresponding Mordell curves $E_B$, using the database in \cite{gh2}. There are no non-torsion integral points with $[n]P$ integral and $n >10 $ on these curves.
\end{proof}

\subsection{Linear form in elliptic logarithms and upper bound on $\mathbf
{n}$}

We use the theory of linear form in elliptic logarithms to derive an upper bound on $n$ such that $[n]P$ is integral. Assume $[n]P$ is integral. Let $ \omega$ be the real period of the elliptic curve $E$, and $z$ be the principal value of the elliptic logarithm of $P$ (the point in related fundamental domain such that $P=(\wp (z), \frac{1}{2} \wp'(z))$. Determine $m$ such that the linear form 
 $$L_{N,m}(z,\omega)=nz+m \omega$$
  is the principal value of the elliptic logarithm $[n]P$. The following lemma \cite{tza} gives an upper bound for this linear form. 
 \begin{lemma}\label{tza}
 Let $Q$ be a rational point on the Mordell curve $E_B$, and $z$ be the principal value of the elliptic logarithm of $Q$. if $x(Q) \geq 2 \sqrt[3]{\Br}$, then
  $$ \log \ghl z \ghr  \leq \frac{3}{2} \log 2 - \frac{1}{2} \log \ghl x(Q) \ghr. $$
  \end{lemma} 
As an immediate result, we have the following lemmas:
\begin{lemma} \label{karan bala}
Let $P$ be a non-torsion point on the quasi-minimal Mordell curve $E_B$. Let $[n]P$ be integral for some $n \geq 11$, and $L_{n,m}(z,\omega)=nz+m \omega$ be the principal value of the elliptic logarithm of $[n]P$, then 
$$ \log \ghl L_{n.m}(z, \omega)\ghr < -n^2 \left(\frac{\log \Br}{36} - C \right)+ \frac{\log \Br}{6} +1.339. $$
Where the value $C$ depends on the congruence conditions of $B$ and is given in Lemma \ref{height}.
 \end{lemma}
 \begin{proof}
To apply Lemma \ref{tza} to $[n]P$, first, we will show that $x([n]P) > 2\Br^{1/3}$. Since $ n \geq 11$
$$ \hat{h}([n]P) \geq 121  \left( \frac{\log \Br}{36} - C \right). $$
Hence, By inequalities \eqref{height2}
$$ \frac{h([n]P)}{2}  > 121  \left( \frac{\log \Br}{36} - C \right) - \frac{\log \Br}{6} -0.299. $$ 
If $h([n]P) < \log (2 \Br^{\frac{1}{3}}) $ then comparing the upper bound  and the lower bound of $h([n]P)$, obtain  
\begin{itemize}
\item[•] If $B \equiv 15120, 3024 \text{ or } 1296 \pmod{15552}$, then $ \ghl B \ghr < 26736, $
\item[•] If $B \equiv 80 \text{ or } 208 \pmod{576}$, then $\ghl B \ghr < 440, $
\item[•]If $B \equiv 13392 \text{ or } 9936 \pmod{15552}$, then $\ghl B \ghr < 440,$
\item[•]If $B \equiv 6372, 2052 \text{ or } 324 \pmod{7776}$, then $\ghl B \ghr < 150$ 
\item[•]If $B \equiv 108 \text{ or } 540 \pmod{3888}$, then  $\ghl B \ghr < 150,$
\item[•]If $B  \equiv 1809 \text{ or } 297 \pmod{1944}$, then  $\ghl B \ghr < 150,$
\item[•]If $B \equiv 144 \pmod{1728}$, then  $\ghl B \ghr < 150,$
\item[•]Otherwise, $\ghl B \ghr < 8.$
\end{itemize}
But, by the same method as the proof of Lemma \ref{2m barabar}, for none of the values $B$ above, quasi-minimal Mordell  curves $E_B$ have a point $P$ where $[n]P$ is integral, with $n>5$. Therefore, $h([n]P) >  \log  \left( 2 \Br^{1/3} \right)$. On the other hand, since $[n]P$ is integral, $h([n]P)= \log \ghl x(P) \ghr$. Since $[n]P$ is a  point on $E_B$, $x(P) > -\Br ^{1/3}$. Therefore, $x \left([n]P \right) > 2 \Br ^{1/3}$. We can now apply Lemma \ref{tza} to $[n]P$, to obtain the result.
  \end{proof}
\begin{lemma} \label{4 karan}
Let $P$ be a non-torsion point on the quasi-minimal Mordell curve $E_B$. Let $[4]P$ be integral, and $L_{4,m}(z,\omega)=4z+m \omega$ be the principal value of the elliptic logarithm of $[4]P$, then 
$$ \log \ghl L_{4,m}(z, \omega) \ghr < -\frac{\log \Br}{3} +0.921. $$ 
\end{lemma} 
\begin{proof}
From equation \eqref{4height}, we have $h([4]P)> \frac{2}{3} \log \Br -0.239 $. For $\Br > 16$ we obtain $x([4]P) > 2 \Br^{1/3} $. Curves $E_B$ with $ \Br \leq 16 $ have no non-torsion point $P$ where $[4]P$ is integral. The result follows from the Lemma \ref{tza}.
\end{proof} 
Combining Remark \ref{2,3 height} and Lemma \ref{tza}, we obtain the following lemma:
\begin{lemma} \label{2,3 karan}
Let $P$ be a non-torsion point on the quasi-minimal Mordell curve $E_B$. Assume $[2]P$ and $[3]P$ are integral, and $L_{3,m}(z,\omega)=3z+m \omega$ is the principal value of the elliptic logarithm of $[3]P$, then 
$$ \log \ghl L_{3,m}(z, \omega) \ghr < -0.22{\log \Br} +\frac{3}{2} \log 2. $$ 
\end{lemma}
Now that we established an upper bound for $\log \ghl L_{n.m} (z,\omega) \ghr$, we prove a lower bound by applying David's lower bounds on linear form in elliptic logarithms \cite{ingram}. 
\begin{lemma}
 Let $E / \mathbb{Q}$ be an elliptic curve, and let $\omega$ and $\omega^{\prime}$ be the real and complex periods of E, chosen such that $\tau=\omega^{\prime} / \omega$ is in the fundamental region
$$
\left\{z \in \mathbb{C}:|z| \geqslant 1, \operatorname{lm}(z)>0 \text {, and }|\operatorname{Re}(z)| \leq \frac{1}{2}\right\}
$$
of the action of $\mathrm{SL}_{2}(\mathbb{Z})$ on the upper half plane. Let $P, z$ and $L_{n, m}$ be defined as in Lemma \ref{karan bala}, and let $B, V_{1}$, and $V_{2}$ be positive real numbers chosen such that
$$
\begin{array}{l}
\log \left(V_{2}\right) \geqslant \max \left\{h(E), \frac{3 \pi}{\operatorname{Im}(\tau)}\right\}, \\
\log \left(V_{1}\right) \geqslant \max \left\{2 \hat{h}(P), h(E), \frac{3 \pi|z|^{2}}{|\omega|^{2} \operatorname{lm}(\tau)}, \log \left(V_{2}\right)\right\}
\end{array}
$$
and
$$
\log (\mathpzc{B}) \geqslant \max \left\{ e h(E), \log |n|, \log |m|, \log \left(V_{1}\right) \right \}.
$$
Then either $L_{\text {n,m }}(z, \omega)=0$, or else
$$
\log \left|L_{n, m}(z, \omega)\right| \geqslant-\mathpzc{C}(\log (\mathpzc{B})+1)(\log \log (\mathpzc{B})+h(E)+1)^{3} \log \left(V_{1}\right) \log \left(V_{2}\right)
$$
where $\mathpzc{C}$ is some large absolute constant (we may take $\mathpzc{C}=4 \cdot 10^{41}$ ).
\end{lemma}

Note that $ L_{\text {n.m }}(z, \omega)=0$ if $P$ is a torsion point.

Let $P$ be a non-torsion point on the quasi-minimal Mordell curve $E_B$ such that $[n]P$ is integral for some $n \geq 11$. Let $L_{n.m}(z,\omega)$ be the principal value of the elliptic logarithm of $[n]P$. Assume that $ \Br > 8 $ so that $h(E)= \log(V_2)= \log (4 \Br )$. We derive a lower bound on $\log \left|L_{n, m}(z, \omega)\right|$ by applying the above lemma. To simplify the outcome bound, consider some cases:

Assume $ 2 \log n < \frac{\log \Br}{3}$, then we can take 
$$ \log (V_2) = \log ( 4 \Br ), \log (V_1)= \log (  4 \Br ) \text{  and  } \log \mathpzc{B}= e \log (4 \Br).$$
So we obtain
\begin{equation}\label{karanpaeinL1}
 \log \left|L_{n, m}(z, \omega)\right| > -5.944 \cdot 10^{42} \log (4 \Br) ^6. 
\end{equation} 

Assume $\frac{\log \Br}{3} < 2 \log n < \left(e -\frac{2}{3} \right) \log \Br$, then we can take 
$$ \log (V_2) = \log ( 4 \Br ), \log (V_1)= e\log (  \Br )-1.199 \text{  and  } \log \Br = e \log (4 \Br).$$
So we obtain
\begin{equation} \label{karanpaeinL2}
\log \left|L_{n, m}(z, \omega)\right| > -5.562 \cdot 10^{42} \log (4 \Br) ^6. 
\end{equation}

Assume $ \left( e -\frac{2}{3} \right) \log \Br < 2 \log n$, then we can take 
$$ \log (V_2) = \log ( 4 \Br ), \log (V_1)= \frac{6 e}{3 e-2} \log n -1.119$$
$$ \text{  and  } \log \Br = e \left( \frac{6}{3 e-2} \log n  +  \log 4 \right). $$
So we obtain:
\begin{equation} \label{karanpaeinL3}
 \log \left|L_{n, m}(z, \omega)\right| > -4.144 \cdot 10^{42} (\log n )^5 (\log 4 \Br).
 \end{equation}

\begin{lemma} \label{karan n}
Let $P$ be a non-torsion point on the quasi-minimal curve $E_B$. Assume $[n]P$ is integral, then 
$$n< \max \{3 \cdot 10^{22} \cdot( \log \Br )^{5/2}, 7.511 \cdot 10^{26} \} .$$
\end{lemma}

\begin{proof}
We can assume $\Br >75 $, and $n >10$. If $ 2 \log n < \left( e -\frac{2}{3} \right) \log \Br$, then by Lemma \ref{karan bala} and inequalities \eqref{karanpaeinL1}, and \ref{karanpaeinL2} we have
% Compare the upper bound for $\left|L_{n, m}(z, \omega)\right|$ in Lemma \ref{karan bala}, \eqref{karanpaeinL1}, 
$$  n^2 \left(\frac{\log \Br}{36} - C \right) <  5.944 \cdot 10^{42} \log (4 \Br) ^6 + \frac{\log \Br}{6} +1.339 $$

For different values of $C$ in \eqref{height}, we can check all finite $B$ values such that $\log( 4 \Br)/\left(\frac{\log{\Br}}{36} - C \right) <150$, none of the corresponding curves $E_B$ has a point $P$, with $[n]P$ integral for $n>5$. So, in this case, assuming $\log( 4 \Br)/\left(\frac{\log{\Br}}{36} - C \right) >150$, we conclude that  $n < 3 \cdot 10^{22} \cdot ( \log \Br )^{5/2}$.

If  $\left( e -\frac{2}{3} \right) \log \Br <  2 \log n $ then by Lemma \ref{karan bala}, and inequality \ref{karanpaeinL3}, we have 
$$  n^2 \left(\frac{\log \Br}{36} - C \right) < 4.144 \cdot 10^{42} (\log n )^5 (\log 4 \Br) + \frac{\log \Br}{6} +1.339 $$
with the same argument as above we can conclude that : $n < 7.511 \cdot(10)^{26}$.

\end{proof}

The next step is to prove a gap principle between larger values of $n$,  where $[n]P$'s are integral.

\begin{lemma} \label{mnasefr}
Let $P$ be a non-torsion point on the quasi-minimal Mordell curve $E_B$, and $z$ be the principal value of elliptic logarithm of the integral point $P$. Assume $[n]P$ is integral and $nz+ m \omega$ is the principal value of $[n]P$. If $n \geq 11$, then $m \neq 0$.
\end{lemma}
\begin{proof}
Again, the idea is to bound the elliptic logarithm of the point $[n]P$, this time assuming $m=0$. If $m=0$, then $\log nz=  \log  z + \log n$. Therefore, from Lemma \ref{karan bala}, we have 
 $$ n^2 \left(\frac{\log \Br}{36} - C \right)- \frac{\log \Br}{6} -1.339 < -\log \ghl z \ghr - \log n,$$ 
 on the other hand, we have: 
 \begin{equation}
 \begin{aligned}
 -\log |z| &=-\log \left|\frac{1}{2} \int_{x_{p}}^{\infty} \frac{d t}{\sqrt{t^{3}-N^{2} t}}\right| \\ & \leqslant \frac{3}{2} \log 2+\frac{1}{2} \log \max \left\{\left|x(P)\right|, 2 \Br ^{1/3}\right\} . 
 \end{aligned}
 \end{equation}
Assume $x(P) >  2 \Br ^{1/3}$, then comparing the bound obtained in  proof of \ref{heightkaran} with the lower bound of $\log nz$, we get:
$$ n^2 \left(\frac{\log \Br}{36} - C \right)< \frac{\log \Br}{3}+\frac{3}{2} \log 2 - \frac{\log 6}{2} +1.339 $$
Which means:
\begin{itemize}
\item[•] If $B \equiv 15120, 3024 \text{ or } 1296 \pmod{15552}$, then $ \ghl B \ghr < 14035, $
\item[•] If $B \equiv 80 \text{ or } 208 \pmod{576}$, then $\ghl B \ghr < 363, $
\item[•]If $B \equiv 13392 \text{ or } 9936 \pmod{15552}$, then $\ghl B \ghr < 363,$
\item[•]If $B \equiv 6372, 2052 \text{ or } 324 \pmod{7776}$, then $\ghl B \ghr < 139$ 
\item[•]If $B \equiv 108 \text{ or } 540 \pmod{3888}$, then  $\ghl B \ghr < 139,$
\item[•]If $B  \equiv 1809 \text{ or } 297 \pmod{1944}$, then  $\ghl B \ghr < 139,$
\item[•]If $B \equiv 144 \pmod{1728}$, then  $\ghl B \ghr < 139,$
\item[•]Otherwise, $\ghl B \ghr < 10.$
\end{itemize}
There are only a few possible values of $B$ in any of the above cases. Fot these values, we have already investigated the multiple of integral points on corresponding curves $E_B$ in Lemma \ref{prime}. As we have seen, integral points $P$ on these curves have no integral multiple $[n]P$, where $n >5$.

If  $x(P) <  2 \Br ^{1/3}$, then 
$$ n^2 \left(\frac{\log \Br}{36} - C \right)< \frac{\log \Br}{3}+2 \log 2 -\log n +1.339. $$
Which again is not possible for $n \geq 11$.
\end{proof}
Now we are ready to prove a gap principle
\begin{lemma} \label{gap}
Let $n_1 < n_2$ be two integers bigger than 10. Assume $P$ is a non-torsion point on the quasi-minimal Mordell curve $E_B$ such that both $[n_1] P$ and $[n_2]P$ are integral, then:
 $$ n_1^2 \left( \frac{\log \Br }{36} -C \right) + \log \omega_1 - \frac{\log \Br}{3} -1.399 < \log(n_2).$$
 Where $C$ is as defined in Lemma \ref{height} and $\omega_1 $ is the real period of $E_1$.
\end{lemma}
\begin{proof}
Let $L_{n_1,m_1}(z,\omega_B),L_{n_2,m_2}(z,\omega_B)$ be the principal values of the elliptic logarithms of $[n_1]P$ and $[n_2]P$. We have 
$$ \omega \ghl n_2m_1- n_1m_2 \ghr \leq n_2 \ghl n_1z+ m_1 \omega \ghr +n_1 \ghl n_2z + m_2 \omega \ghr. $$
Assume  $n_2m_1- n_1m_2=0$, then since $n_2$ and $m_2$ are nonzero, we have $ n_2 \mid m_2 n_1$,  $n_1$ and $n_2$ are both prime, so $n_2 \mid m_2$, but it is easy to see that $m_2 \leq \frac{1}{2}(n_2+1)$. Since by Lemma \ref{mnasefr} $m_2 \neq 0$, $\ghl n_2m_1- n_1m_2 \ghr \geq 1$. Therefore,  from Lemma \ref{karan bala}, we obtain
\begin{equation*}
\begin{split}
 \Br ^{-1/6} \omega_1 \leq n_2\exp \left( -n_1^2 \left(\frac{\log \Br}{36} - C \right)  +\frac{\log \Br}{6} +1.339 \right) \\  +  n_1 \exp \left(-n_2^2 \left(\frac{\log \Br}{36} - C \right)+ \frac{\log \Br}{6} +1.339 \right). 
\end{split}
\end{equation*}
With the assumption $\Br >75$ and the congruence relations, we can see that the first summand on the left-hand side is bigger than the second one, and we obtain :
$$ \frac{1}{2} \Br ^{-1/6} \omega_1 \leq n_2\exp \left( -n_1^2 \left(\frac{\log \Br}{36} - C \right)  +\frac{\log \Br}{6} +1.339 \right). $$
The results follow by taking logarithms of both sides. 
\end{proof}
\begin{remark}
Note that in the proof, we only used $m_2 \neq 0$, so the above argument also works when $m_1 =0.$ 
\end{remark}
We are now ready to prove Theorem \ref{main2}. Let $P$ be a non-torsion point on the quasi-minimal Mordell curve $E_B$ such that both $[n_1]P $ and $[n_2]P$ are integral, with $n_1, n_2 \geq 11$. Combining Lemma \ref{karan n} with the gap principle in Lemma \ref{gap}, we obtain
\begin{equation} \label{gap 2}
\begin{split}
 n_1^2 \left( \frac{\log \Br }{36} -C \right) +\log \omega_1 - \frac{\log \Br}{3} -1.399 < \log n_2 \\
 < \log \max \left( 7.511 \cdot 10^{26}, 3 \cdot 10^{22}(\log \Br)^{\frac{5}{2}} \right).
 \end{split} 
\end{equation}
Let $n_1 \geq 29$ then we have:
\begin{itemize}
\item[•] If $B \equiv 15120, 3024 \text{ or } 1296 \pmod{15552}$, then $ \ghl B \ghr < 56776, $
\item[•] If $B \equiv 80 \text{ or } 208 \pmod{576}$, then $\ghl B \ghr < 2791, $
\item[•]If $B \equiv 13392 \text{ or } 9936 \pmod{15552}$, then $\ghl B \ghr < 2791,$
\item[•]If $B \equiv 6372, 2052 \text{ or } 324 \pmod{7776}$, then $\ghl B \ghr < 846$ 
\item[•]If $B \equiv 108 \text{ or } 540 \pmod{3888}$, then  $\ghl B \ghr < 846,$
\item[•]If $B  \equiv 1809 \text{ or } 297 \pmod{1944}$, then  $\ghl B \ghr < 846,$
\item[•]If $B \equiv 144 \pmod{1728}$, then  $\ghl B \ghr < 846,$
\item[•]Otherwise, $\ghl B \ghr < 71.$
\end{itemize}
As we have already checked the case $\Br <75$, any of the above conditions holds only for a few $B$ values. with a similar argument as the proof of Lemma \ref{2m barabar}, Based on database in \cite{gh2} None of the corresponding Mordell curves $E_B$ has an integral point $P$ with $[n]P$ integral and $n>5.$ This completes the proof of Theorem \ref{main2}.

With exactly the same argument, with $n_1=11$, we can also prove the following lemma for larger values of $\Br $:
\begin{lemma} \label{n>11}
Let $ \Br > 6.2675 \cdot 10^{12}$, and $P$ be a non-torsion point on the quasi-minimal Mordell curve $E_B$. Then there is at most one value $n>10$ such that $[n]P$ is integral.
\end{lemma}

\section{proof of Theorem \ref{main 3}} \label{sec5}
Let $P$ be a torsion point on the quasi-minimal curve $E_B$. In this section, we will show that $P$ has at most three integral multiples $[n]P$, with $n>1$. So far, we have talked about the number of integral multiples of $P$ when $n \leq 10$ (Theorem \ref{main1}), when $n \geq 29$ (Theorem \ref{main2}), and also when $n \geq 11$ for larger values of $B$ (Lemma \ref{n>11}). 

For values $ 10 < n <29$, we prove the following lemma.
\begin{lemma} \label{n11,29}
Let $P$ be a non-torsion point on the quasi-minimal curve $E_B$. Then $[n]P$ is integral for at most one value $10 < n <29.$
\end{lemma}
\begin{proof}
 Assume $[n_1]P$ and $[n_2]P$ be integral for two distinct values of $ 10 < n_1 < n_2 <29$. By Lemma \ref{prime}, $n_1$ and $n_2 $ are both prime so they are both less than 24. We can apply the gap principle in Lemma \ref{gap} for $n_1$, and $n_2$ to obtain the following inequality
 $$ 121 \left( \frac{\log \Br }{36} -C \right) +\log \omega_1 - \frac{\log \Br}{3} -1.399 < \log 23.$$
 Therefore if both $[n_1]P$ and $[n_2]P$ are integral, we have
 \begin{itemize}
\item[•] If $B \equiv 15120, 3024 \text{ or } 1296 \pmod{15552}$, then $ \ghl B \ghr < 23456, $
\item[•] If $B \equiv 80 \text{ or } 208 \pmod{576}$, then $\ghl B \ghr < 606, $
\item[•]If $B \equiv 13392 \text{ or } 9936 \pmod{15552}$, then $\ghl B \ghr < 606,$
\item[•]If $B \equiv 6372, 2052 \text{ or } 324 \pmod{7776}$, then $\ghl B \ghr < 233$ 
\item[•]If $B \equiv 108 \text{ or } 540 \pmod{3888}$, then  $\ghl B \ghr < 233,$
\item[•]If $B  \equiv 1809 \text{ or } 297 \pmod{1944}$, then  $\ghl B \ghr < 233,$
\item[•]If $B \equiv 144 \pmod{1728}$, then  $\ghl B \ghr < 233,$
\item[•]Otherwise, $\ghl B \ghr < 16.$
\end{itemize} 
But we have already seen that non-torsion points on the quasi-minimal Mordell curves $E_B$ with above values of $B$ do not have any integral multiples $[n]P$, with $ n>5 $.
\end{proof}
 To summarize, by Theorems \ref{main1} and \ref{main2}, and Lemmas \ref{prime}, \ref{n>11} and \ref{n11,29}, Theorem \ref{main 3} holds in the following cases:(I) when $\Br > 6.2675 \cdot 10^{12}$, (II) when $P$ has either fewer or more than two integral multiple $[n]P$ with $ 2 \leq n \leq 10$, (III) $P$ has no integral multiple $[n]P$ with $ 11 \leq n \leq 23$. To complete the proof, it is sufficient to 
prove that if $P$ is a non-torsion point on the quasi-minimal Mordell curve $E_B$, with $\Br < 6.2675 \cdot 10^{12}$, and two integral multiples $[n]P$, $1<n<11$, then $[n]P$ is not integral for any value $11 \leq n \leq 23$.  

By Theorem \ref{main1} and Lemma \ref{5 barabar}, if $P$ has two integral multiples $[n]P$, with $1<n<11$, then either $[2]P$ and $[3]P$ are integral or $[4]P$ is integral.

Assume $[2]P$ and $[3]P$ are integral. Recall proposition \ref{2,3}, in which we categorized all the points $P$ with $[2]P$ and $[3]P$ are integral. in each family of quasi-minimal curves $E_B$ with $\Br < 6.2675 \cdot 10^{12}$, there are only finitely many such curves. We determine the points $P$ on these curves for each family and calculate $[n]P$ for $n=11, 13, 17,19,23$. (By Lemma \ref{prime}, we only need to check prime values of $n$.) We used sage to perform these calculations. based on our calculations, there is no integral multiple among the calculated points.

Finally, assume $[4]P$ is integral and $ \Br < 6.2675 \cdot 10^{12}$. All the points with $[4]P$ integral can be determined by solving some pell equations associated with equations in \eqref{4barbar categorize}. We look for solutions corresponding to Pell's equations with the restriction $ \Br < 6.2675 \cdot 10^{12}$. All the non-torsion integral points on the quasi-minimal curve $E_B$ with this restriction are given in Table \ref{jadval 2}. None of these points have integral multiple $[n]P$ with a prime integer $n$, $ 11 \leq n \leq 23$. This completes the proof of Theorem \ref{main 3}.
\begin{table}[tb]
    \centering
    \caption{Integral points $P$ on elliptic curves $E_B$ with $[4]P$ integral and $ \Br < 6.2675 \cdot 10^{12}$.}
    \label{jadval 2}

\begin{center}
\begin{tabular}{ |c|c| }  
 \hline
   $B$    &  $P$     \\
 \hline         
-13500 & $[60,450]$   \\
\hline
-21168 & $[84,756]$  \\
\hline
-2743600 & $[380,7220]$  \\
\hline
80 & $[4,12]$  \\
\hline
-1124695 & $[286,4719]$  \\
\hline
59400 & $[60,900]$  \\
\hline
-83338860528 & $[11928,1270332]$  \\
\hline
513 & $[6,27]$  \\
\hline
-197137217456 & $[15908,1956684]$  \\
\hline
30371652 & $[228,6498]$  \\
\hline
74088784 & $[308,10164]$  \\
\hline
7301384400 & $[1420,100820]$  \\
\hline
5322709227600 & $[12780,2722140]$  \\
\hline
3086626985 & $[1066,6559]$  \\
\hline
\end{tabular}
\end{center}
\end{table}   

\begin{remark}
An alternative approach to prove Theorem \ref{main 3} is using the gap principle ( Lemma \ref{gap}) and Lemmas \ref{4 karan}, \ref{2,3 karan} to find a lower bound on the smallest integral multiple bigger than 4 of $P$, assuming $[4]P$ or $[2]P$ and $[3]P$ are integral. One can prove such a result for values $B$ bigger than some constant and treat the remaining cases as above.   
\end{remark}
\section{Integral points on Mordell curves of rank 1}
\subsection{Torsion-free quasi-minimal Mordell curves}
Let $E_B$ be a torsion-free quasi-minimal Mordell curve of rank 1. From Theorem \ref{main 3}, it is clear that $E_B$ has at most 8 integral points. This bound is sharp in the sense that all the curves $E_{108}$, $E_{80}$, $E_{-13500}$ and $E_{-21168}$ mentioned in Theorem \ref{main1} are quasi-minimal Mordell curves of rank 1 with eight integral points. However, as we will explain, it is plausible to believe that these are the only torsion-free rank 1 quasi-minimal Mordell curves  with more than six integral points. Let $B$ be a sixth-power-free integer not equal to 108, 80, or -13500, and $E_B$ be a torsion-free Mordell curve of rank 1 with more than six integral points. Let $P$ be a generator of $E_B$'s Mordell-Weil group. Then $P$ has three integral multiples $[n]P$ with $n>1$. Let $N$ be the largest integer such that $[n]P$ is integral. By Theorems \ref{main1}, \ref{main2}, Lemmas \ref{n>11}, \ref{n11,29}, and the argument in the proof of Theorem \ref{main 3}, $N \geq 29$ if $ \Br < 6.2675 \cdot 10^{12}$, and $ N \geq 11$ otherwise. Since $[N]P$ is integral, from the inequalities \eqref{heightkoli} and \eqref{height2}, we obtain:
\begin{equation} \label{Large}
\frac{\log \left(x([n]P)\right)}{2} >  N^2 \left(\log \frac{\Br}{36}-0.2262 \right)- \frac{1}{6} \log \Br  -0.299. 
\end{equation}
On the other hand, regarding upper bounds on the height of integral points on Mordell curves, we have the following well-known conjecture of Marshal Hall. 

\begin{conjecture} (Hall)  Given $\epsilon >0$, there exists a positive constant $C_\epsilon$ such that the inequality  
$$ 
\left\lvert x(P) \right\rvert < C_\epsilon |B|^{2+\epsilon} 
$$
holds for all integral points $P$ on the Mordell curve $E_B$.
\end{conjecture}
In \cite{gh2} the authors list all the Mordell curves encountered with Hall Measure $x(P)^{1/2}/|B| $ exceeding $1$; with the biggest value being less than 1.41. There are some lists of integral points on Mordell curves with relatively large value $x(P)^{1/2}/|B| $ \cite{Elk, CHS}. The biggest known record is about 46.60, and the second biggest record is less than 17. However, If $B>8000$, and the inequality \eqref{Large} holds, the value of $x([n]P)^{1/2}/|B|$ would be bigger than 7575. So it is likely that any torsion-free rank one subgroup of rational points on Mordell curves $E_B$ bar the mentioned exceptions have at most six integral points. Note that we introduce infinitely many examples of quasi-minimal Mordell curves with six integral points on a torsion-free subgroup of rank 1 of their Mordell-Weil group (Proposition \ref{2,3}, and Corollary \ref{binahayat 2,3}). From the above argument in connection to Theorem \ref{main1}, we can expect that a point $P$ on a quasi-minimal Mordell curve $E_B$, at least for larger values of $\Br$, has no integral multiples $[n]P$, with $n >5$. However, the methods we used in this paper only gives a reasonable upper bound on the second largest integral multiple of a point $P$ on the quasi minimal Mordell curve $E_B$. The largest integral multiple is still based on linear form in logarithms which is far from expected. 

\subsection{quasi-minimal Mordell curves with $\mathbb{Z}/ 2 \mathbb{Z}$ as the torsion subgroup } \label{integral points torsion}
Let $B$ be a sixth-power-free integer. Assume the Mordell curve $E_B$ has a rational torsion subgroup isomorphic to $\mathbb{Z}/ 2 \mathbb{Z}$. By Lemma \ref{torsionknapp}, $B=B_0^3$ is a cube different from 1. Since $B \not \equiv 16 \pmod {64}$, the equation $y^2=x^3+B$ is a global minimal model for $E_B$ over the field of rational numbers. From Remark \ref{torsionknapp2}, $T=(-B_0,0)$ is the non-trivial torsion point on $E_B$.

\begin{lemma}\label{zarib torsion}
Let $B$ be a sixth-power-free cube integer not equal to 1, and $P$ be a non-torsion point on $E_B$. 
\begin{itemize}
\item[a)] If $\Br > 5.333 \cdot 10^5$ is odd, then $[n]P$ is integral for at most two values $n>1$.
\item[b)]  If $\Br > 1408$ is even, then $[n]P$ is integral for at most one value $n>1$. 
\end{itemize}
\end{lemma}
\begin{proof}
By Theorem \ref{main1} and Lemmas \ref{5 barabar}, \ref{3torsion} and  \ref{torsion4}, $[n]P$ is not integral for any $3 \leq n \leq 10$. Moreover, if $B$ is even, from \eqref{2 barabar}, $[2]P$ is not integral. Now assume $[n]P$ be integral for two values $n_1, n_2 $ both bigger than 10. Note that by Theorem \ref{main1}, $n_1$, $n_2$ are both odd. By Lemma \ref{gap 2}, we have 
 \begin{equation*} 
\begin{split}
 11^2 \left( \frac{\log \Br }{24} -C \right) +\log \omega_1 - \frac{\log \Br}{3} -1.399 < \log n_2 \\
 < \log \max \left( 7.511 \cdot 10^{26}, 3 \cdot 10^{22}(\log \Br)^{\frac{5}{2}} \right),
 \end{split} 
\end{equation*}
where from inequalities \eqref{heightcube}, we can take $C=0.002$ when $B$ is odd, and $C=-0.2290$, when $B$ is even. The above inequality holds when $B$ is odd and $ \Br < 5.333 \cdot 10^5$, or when $B$ is even and $\Br < 1409.$ 
\end{proof}

 We recall a result of Ayad \cite{ayad} to show that for any point $P$ on the minimal curve $E_B$, $[4]P+T$ is not integral. This results holds for elliptic curves in global minimal model.

\begin{lemma}
Let $E/ \mathbb{Q}$ be an elliptic curve, $M =(\frac{a}{d^2},\frac{b}{d^3})$ be a non-torsion rational point on $E$, and $L = (u,v)$ be a rational 2-torsion point on $E$. Assume the set $S$ contains all rational primes $p$ where $M \mod p$ is singular. Let $m \in \mathbb{Z}; m \neq 0$, $L +mM $ is an S-integer, if and only if 
 $$ \widehat{\phi_m} -u d^2 \widehat{\psi_m^2}= \pm \prod_{p \in S} p^{e_p} \quad \text{if $ u \in \mathbb{Z}$} $$
 and 
  $$ 4(\widehat{\phi_m} -u d^2 \widehat{\psi_m^2})= \pm \prod_{p \in S} p^{e_p} \quad \text{if $ u \notin \mathbb{Z}$}, $$
 where $\widehat{\phi_m}=d^{2m^2} \phi_m(M)$ and $\widehat{\psi_m}=d^{m^2-1} \psi_m(M)$.
\end{lemma}  

\begin{lemma} {\label{4,T,1}}
Let $P$ be non-torsion rational point on the Mordell curve $E_B$, where $B=B_0^3$ is a sixth-power-free cube, and $T=(-B_0,0)$. Then $[4]P+T$ is not an integral point.
\end{lemma}
\begin{proof}
 Let $P=\left( \frac{a}{d^2}, \frac{b}{d^3} \right)$. Note that $T$ is a rational 2-torsion point on $E_B$. From Theorem A. of \cite{ayad}, $P \pmod p$ is singular if and only if $p=2$ or $p$ divides both $B$ and $a$. Applying the above lemma and \eqref{divpol} if $[4]P+T$ is integral, then we have :
 
 \begin{equation} \label{4,T}
 \begin{aligned} 
 &f(x,y)= \\
 &\left( x^8 + 8x^7y-32x^6y^2-16x^5y^3-56x^4y^4-64x^3y^5+64x^2y^6
  -64xy^7-32y^8\right)^2 \\
 & = \pm \prod_{p \mid 2B_0} p^{e_p}, 
 \end{aligned} 
   \end{equation}
 where $x=a$, and $y= B_0d^2$. Equation \eqref{4,T} is a Thue-Mahler equation which has finitely many solutions. Here the goal is to reduce it to finitely many Thue equations.
  Let $ \mathcal{A}= \frac{a}{\gcd(a,B_0d^2)}$ and $\mathcal{B}=\frac{B_0d^2}{\gcd(a,B_0d^2)}.$
  
   Let $ p > 3$ be a divisor of $B_0d^2$ and $a$, If $\ord_p(a) \neq \ord_p(Bd^2)$ then clearly $ p \nmid f(\mathcal{A},\mathcal{B} )$. Assume $\ord_p(a) = \ord_p(Bd^2)$ so that $\ord_p(\mathcal{A}) = \ord_p(\mathcal{B})=0.$ Then since $b^2=a^3+(B_0d^2)^3$, we have $\ord_p ( \mathcal {A}^3 + \mathcal{B}^3 ) \geq 1$.
  Let $g(x,y)= x^5+8yx^4-32y^2x^3-17y^3x^2-64y^4x-32y^5$ then 
 $$ f \left(\mathcal{A}, \mathcal{B} \right)- g \left(\mathcal{A}, \mathcal{B} \right) \left( \mathcal {A}^3 + \mathcal{B}^3 \right)= 81 \mathcal{A}^6 \mathcal{B}^2.$$
 Therefore, unless $p=3$, in this case, $ p \nmid f(\mathcal{A},\mathcal{B} ).$ 
 If $p=3$ and $\ord_3(\mathcal{A}) = \ord_3(\mathcal{B})=0.$ Then $\ord_3 ( \mathcal {A}^3 + \mathcal{B}^3 ) \geq 3$ and $\ord_3(g \left(\mathcal{A}, \mathcal{B} \right)) \geq 2 $, hence $\ord_3 ( f \left(\mathcal{A}, \mathcal{B} \right))=4$
  
If $p=2$, then it is easy to see that $\ord_2 (f(\mathcal{A},\mathcal{B} ))= 0$, when $\ord_2(\mathcal{A})=0$ or $ \ord_2 (f(\mathcal{A},\mathcal{B} ))=5$, when $\ord_2(\mathcal{A}) > 0$. 

Hence the solutions of \eqref{4,T} embed into the solutions of the Thue equations $f(x,y)= \pm 2^{\alpha} 3^{\beta}$ where $ \alpha\in\{0,5\} $ and $\beta\in\{0,4\}$. We check the solutions of these Thue equations using Magma. The only solutions are $( 0 , \pm 1 )$, $ ( \pm1 , 0)$, $(\pm 1, \pm1 )$, and $( \pm 2, \pm1 )$ which correspond to the torsion points. So there is no rational non-torsion point $P$ on the curves $E_B$, where $[4]P+T$ is integral. 
 
\end{proof} 
\begin{remark}
The same steps for $[2]P+T$ lead to some Pellian equations with infinitely many integral solutions.
\end{remark} 

\begin{lemma}
Let $E_B$ be a minimal Mordell curve of rank 1 with a rational torsion subgroup isomorphic to $\mathbb{Z}/2 \mathbb{Z}$. Then $E_B$ has at most 14 integral points when $B$ is odd and at most 12 integral points when $B$ is even.
\end{lemma}
\begin{proof}

%\begin{itemize}
%\item[a)] If $\Br > 5.4871 \cdot 10^9 $ is odd, then $E_B$ has at most 14 integral points.
%\item[b)]  If $\Br > 4.1628 \cdot 10^7$ is even, then  $E_B$ has at most 12 integral points. 
%\end{itemize}

 First assume $\Br <10^6$. We directly checked the number of integral points on corresponding quasi-minimal Mordell curves $E_B$ using the data in \cite{gh2}. For elliptic curves $E_B$ with rank 1, and $B$ even in this range, the elliptic curve $E_8$ has the most number of integral pints with 7 integral points. For elliptic curves $E_B$ with rank 1, and $B$ odd in this range, the elliptic curve $E_{-343}$ has the most number of integral pints with 9 integral points. So we might assume $\Br >10^6$.

By the Assumption, $B=B_0^3$ is a sixth-power free cube not equal to 1.According to Remark \ref{torsionknapp2}, $T=(-B_0,0)$ is a non-trivial torsion point on $E_B$. Let $P$ be a generator of its Mordell weil group. The group of rational points is generated by $P$ and $T$. 

By Lemmas \ref{zarib torsion} and \ref{4,T,1}, when $\Br >  5.33 \cdot 10^5$ and odd, the set of integral points on $E_B$ is a subset of
$$ \{ \pm P, \pm (P+T) , \pm [2]P , \pm ([2]P+T) , \pm [m_1]P , \pm ([m_2]P+T) , \pm ([2m_3]P+T) \}.$$
Here $m_i$'s are odd integers larger than 10. When $\Br > 1409$ is even, the set of integral points is a subset of 
$$ \{ \pm P, \pm (P+T)  , \pm ([2]P+T) , \pm [m_1]P , \pm ([m_2]P+T) , \pm ([2m_3]P+T) \}.$$
Again $m_i$'s are odd integers larger than 10, these curves have at most twelve integral points.

\end{proof}

% Elkies \cite{Elk} (note that the case with $k=-852135$ was omitted from this paper due to a transcription error) and in work of Jim\'enez Calvo,  Herranz and S\'aez \cite{CHS}.


\subsection*{Acknowledgements}
I would like to thank the anonymous reviewer for the great suggestions and for pointing out numerous mistakes and typos in the earlier version of this manuscript. This research was in part supported by a grant from IPM (No.99110035).

%%%%%%%%%%% To ease editing, use normal size for the references:

\section{Appendix} \label{appendix}
 The main objective of this section is to present data that enables determining a lower bound for canonical height of a rational non-torsion point $P$ on the quasi-minimal Mordell curve $E_B$, solely based on $B$ modulo powers of 2 and 3. Subsequently, we proceed to complete the proof of Lemmas \ref{height} and \ref{heightcube}. 
 
For values of  $\Br \leq 37$, the PARI function "ellgenrators" returns the generators of the Mordell-weil group of $E_B(\mathbb{Q})$. We utilize this function to establish a lower bound for the canonical height of rational points on curves within this range. For any non-torsion rational point $P$ on $E_B$, with $\Br \leq 37$, we find that $\hat{h}(P)- \frac{1}{36} \log( \Br) > 0.0803$. Hence, Lemma \ref{height} holds for $E_B$ when $\Br \leq 37$. Therefore, throughout this section, we can assume that $\Br >37$. 

In each of the following subsections, we use Theorem 1.2 and Tables 4 and 5 of \cite{vou} to determine a lower bound of $\hat{h}(P)$ for a rational non-torsion point $P$ on $E_B$ based on $B$ modulo powers of 2 and 3 and the Tamagawa index at $p$ for a prime $p$. Subsequently, we compare the calculated bounds with the bounds in Lemma \ref{height}. In any case, there are only a few curves that the calculated bounds are superior to those indicated in Lemma \ref{height}. For rational points on these few curves, we show directly that Lemma \ref{height} holds. 

\subsection{Curves with $c_p=1$ for all primes, $p>3$}
In Table \ref{c_p=1}, we consider the rational points on curves $E_B$ with Tamagawa index 1 at every prime $p>3$. The first column represents the congruence conditions of $B$. The second column displays the lower bound for canonical height of non-torsion rational points $P$ on $E_B$ based on part (a) of Theorem 1.2 and Tables 4 and 5 of \cite{vou}. The third column indicates $M$, which is the integer such that for all values $\Br > M$, the stated lower bound in Lemma \ref{height} is smaller than the one in the second column. The fourth column indicates the values $37 <\Br <M$ that satisfy the congruence conditions in column 1, and the last column shows the rank of $E_B$ for the corresponding values $B$ in the fourth column.

\begin{table}[tb]
    \centering
    \caption{Rational points on curves $E_B$ with Tamagawa index 1 at every prime $p>3$}
    \label{c_p=1}
    \begin{tabular}{ |p{3cm}|p{3cm}|c|p{4cm}|c| }  
        \hline
        Congruence conditions of $B$ & Lower bound for $\hat{h}(P)$ on $E_B$ & $M$ & Values $37 < \Br < M$ that satisfy the congruence conditions in column 1 & Rank of $E_B$ \\
        \hline         
        13392, 9936 mod 15552 & $\frac{1}{6} \log \Br - 1.2921$ & 4160 & -2160 & 1 \\
        \hline
        1296 mod 15552 & $\frac{1}{6} \log \Br - 1.2006$ & 1113 & - & - \\
        \hline
        15120, 3024 mod 15552 & $\frac{1}{6} \log \Br - 1.2006$ & 1113 & -432 & 0 \\
        \hline
        4644, 1188 mod 7776 & $\frac{1}{6} \log \Br - 1.0611$ & 1524 & 1188 & 1 \\
        \hline
        324 mod 7776 & $\frac{1}{6} \log \Br - 0.9695$ & 484 & 324 & 0  \\
        \hline
        2052, 6372 mod 7776 & $\frac{1}{6} \log \Br - 0.9695$ & 484 & - & -  \\
        \hline
        3132, 2700 mod 3888 & $\frac{1}{6} \log \Br - 0.9455$ & 663 & - & -  \\
        \hline
        216, 3672 mod 3888 & $\frac{1}{6} \log \Br - 0.9455$ & 663 & 216, -216 & 0, 1  \\
        \hline
        2268 mod 3888 & $\frac{1}{6} \log \Br - 0.854$ & 343 & - & -  \\
        \hline
        108, 540 mod 3888 & $\frac{1}{6} \log \Br - 0.854$ & 210 & 108 & 1  \\
        \hline
        3240 mod 3888 & $\frac{1}{6} \log \Br - 0.854$ & 343 & - & -  \\
        \hline
        1080,1512 mod 3888 & $\frac{1}{6} \log \Br - 0.854$ & 343 & - & -  \\
        \hline
        144 mod 1728 & $\frac{1}{6} \log \Br - 0.8344$ & 183 & 144 & 0  \\
        \hline
        513, 945 mod 1944 & $\frac{1}{6} \log \Br - 0.82997$ & 288 & -& -\\
        \hline
         80,208 mod 576 & $\frac{1}{6} \log \Br -0.7428 $ & 79& -& -  \\
        \hline
         1809,297 mod 1944 & $\frac{1}{6} \log \Br -0.7385 $ & 91& -& -  \\
        \hline
         81 mod 1944 & $\frac{1}{6} \log \Br -0.7385 $ & 149& 81& 0  \\
        \hline
         36 mod 864 & $\frac{1}{6} \log \Br -0.6033 $ & 56& -& -  \\
         \hline
Otherwise & $\frac{1}{6} \log \Br -0.5118 $ & 29& -&  - \\
\hline
\end{tabular}
%\end{center}
\end{table}
 
Based on the data in Table \ref{c_p=1}, to prove Lemma \ref{height}, in this case, we only need to check the height of rational points on the curves $E_B$ with $B \in \{-2160, 1188,-216,108 \}$. Note that all these curves have rank 1. 

The Mordell-weil group of  $E_{-2160}(\mathbb{Q})$ is a torsion-free group of rank 1, with $P=(24,108)$ as a generator. Using PARI, $\hat{h}(P)>\frac{1}{36} \log(2160)+0.01718 > \frac{1}{36}\log (2160) -0.1347 $.

The Mordell-weil group of  $E_{1188}(\mathbb{Q})$ is a torsion-free group of rank 1, with $P=(12,54)$ as a generator. Using PARI, $\hat{h}(P)>\frac{1}{36} \log(1188)+0.0476 > \frac{1}{36}\log (1188)  -0.0431 $.

The Mordell-weil group of  $E_{-216}(\mathbb{Q})$ is a group of rank 1, with torsion subgroup ismorphic to $\mathbb{Z}/2\mathbb{Z}$ and  $P=(10,28)$ as a generator. Using PARI, $\hat{h}(P)>\frac{1}{36} \log(216)+ 0.8305 > \frac{1}{36}\log (216) -0.0431 $.

The Mordell-weil group of  $E_{108}(\mathbb{Q})$ is a torsion-free group of rank 1, with $P=(6,18)$ as a generator. Using PARI, $\hat{h}(P)>\frac{1}{36} \log(108)-0.0551 > \frac{1}{36}\log (108) -0.1107 $.

Therefore the Lemma \ref{height} holds in this case.

\subsection{Curves with $c_p \mid 4$ for all primes, $p > 3$ and $2 \mid c_p$ for at least one such prime}
In Table \ref{c_p2}, we consider the rational points on curves $E_B$ with Tamagawa index $c_p \mid 4$ for all primes, $p > 3$ and $2 \mid c_p$ for at least one such prime. The first column represents the congruence conditions of $B$. The second column displays the lower bound for canonical height of non-torsion rational points $P$ on $E_B$ based on part (b) of Theorem 1.2 and Tables 4 and 5 of \cite{vou}. The third column indicates $M$, which is the integer such that for all values $\Br > M$, the stated lower bound in Lemma \ref{height} is smaller than the one in the second column. The fourth column indicates the values $37 <\Br <M$ that satisfy the congruence conditions in column 1. Finally, the last column lists the curves $E_B$ with positive rank, that satisfy the condition $2 \mid c_p$ for at least one prime $p$, where the corresponding values of $B$ are provided in the fourth column.
\begin{table}[tb]
    \centering
    \caption{Rational points on curves $E_B$ with $c_p \mid 4$ for all primes, $p > 3$ and $2 \mid c_p$ for at least one such prime}
    \label{c_p2}
    \begin{tabular}{ |p{3cm}|p{3.2cm}|c|p{4cm}|p{4 cm}| }  
        \hline
        Congruence conditions of $B$ & Lower bound for $\hat{h}(P)$ on $E_B$ & $M$ & Values $37 < \Br < M$ that satisfy the congruence conditions in column 1 & Positive rank curves $E_B$ with $2$ divides $c_p$ for at least one prime $p$  \\
        \hline         
        1296 mod 15552 & $\frac{1}{24} \log \Br - 0.3007$ & 213 & - & - \\
        \hline
        15120,3024 mod 15552 & $\frac{1}{24} \log \Br - 0.3007$ & 213 & - & - \\
        \hline
        324 mod 7776 & $\frac{1}{24} \log \Br - 0.2429$ & 13607 & -7452, 324, 8100 & - \\
        \hline
        2052,6372 mod 7776 & $\frac{1}{24} \log \Br - 0.2429$ & 13607 &   -13500, -9180,-5724, -1404, 2052, 6372, 9828 & -13500 \\
        \hline
         144 mod 1728 & $\frac{1}{24} \log \Br - 0.2091$ & 1193 & 144 & - \\
         \hline
          1809,297 mod 1944 & $\frac{1}{24} \log \Br - 0.1852$ & 213 & -135 & - \\ 
          \hline
         81 mod 1944 & $\frac{1}{24} \log \Br - 0.1852$ & 27755 & $\{-1944 \dot i +81: -14 \leq i \leq 14 \}$ & - \\  
         \hline
         36 mod 864 & $\frac{1}{24} \log \Br - 0.1514$ & 2434 & -1692, -828, 36, 900, 1764 &-  \\
         \hline
         2268 mod 3888 & $\frac{1}{24} \log \Br - 0.1274$ & 432 & - &-  \\
         \hline
         13392,9136 mod 15552 & $\frac{1}{24} \log \Br - 0.1176 $ & - & - &-  \\ 
         \hline
         108,80 mod 15552 & $\frac{1}{24} \log \Br - 0.1176 $ & - & - &-  \\
         \hline
         Otherwise & $\frac{1}{24} \log \Br - 0.0935$& 37 & - &-  \\  
         \hline
\end{tabular}
%\end{center}
\end{table}

Based on the data in Table \ref{c_p2}, to prove Lemma \ref{height}, in this case, we only need to check the height of rational points on the curves $E_{-13500}.$ 

The Mordell-weil group of  $E_{-13500}(\mathbb{Q})$ is a torsion-free group of rank 1, with $P=(60,450)$ as a generator. Using PARI, $\hat{h}(P)>\frac{1}{36} \log(13500)-0.09730 > \frac{1}{36}\log (13500) -0.1107 $.

Therefore, Lemma \ref{height} holds in this case.

\subsection{Curves $c_p \mid 3$ for all primes, $p > 3$ and $c_p=3$ for at least one such prime,}
In Table \ref{c_p3}, we consider the rational points on curves $E_B$ with Tamagawa index $c_p \mid 3$ for all primes, $p > 3$ and $c_p=3$ for at least one such prime. The first column represents the congruence conditions of $B$. The second column displays the lower bound for canonical height of non-torsion rational points $P$ on $E_B$ based on part (c) of Theorem 1.2 and Tables 4 and 5 of \cite{vou}. The third column indicates $M$, which is the integer such that for all values $\Br > M$, the stated lower bound in Lemma \ref{height} is smaller than the one in the second column. The fourth column indicates the values $37 <\Br <M$ that satisfy the congruence conditions in column 1. Finally, the last column lists the curves $E_B$ with positive rank, that satisfy the condition $3 \mid c_p$ for at least one prime $p$, where the corresponding values of $B$ are provided in the fourth column.

\begin{table}[tb]
    \centering
    \caption{Rational points on curves $E_B$ with $c_p \mid 3$ for all primes, $p > 3$ and $c_p=3$ for at least one such prime}
    \label{c_p3}
    \begin{tabular}{ |p{3cm}|p{3.2cm}|c|p{4cm}|p{4 cm}| }  
        \hline
        Congruence conditions of $B$ & Lower bound for $\hat{h}(P)$ on $E_B$ & $M$ & Values $37 < \Br < M$ that satisfy the congruence conditions in column 1 & Positive rank curves $E_B$ with $3$ divides $c_p$ for at least one prime $p$  \\
        \hline         
        13392,9936 mod 15552 & $\frac{1}{18} \log \Br - 0.4182$ & 27065 & -21168, -17712, -5616, -2160, 9936, 13392, 25488 & -21168  \\
        \hline
        15120,3024 mod 15552 & $\frac{1}{18} \log \Br - 0.3267$ & 37&- & -  \\
         \hline
        216,3672 mod 3888 & $\frac{1}{18} \log \Br - 0.3027$ & 11448 & $\pm 11448, \pm 7992$, $\pm 7560$, $\pm 4104$, $\pm 3672$, $\pm 216$  & -  \\
        \hline
        2700,3132 mod 3888 & $\frac{1}{18} \log \Br - 0.3027$ & 11448 & -8964, -8532, -5076, -4644, -1188, -756, 2700, 3132, 6588, 7020, 10476, 10908& -\\
        \hline
        80,208 mod 576 & $\frac{1}{18} \log \Br - 0.2351$ & 37&- & -  \\
        \hline
        108,540 mod 3888 & $\frac{1}{18} \log \Br - 0.2111$ & 37&- & -  \\
        \hline
        1080,1512 mod 3888 & $\frac{1}{18} \log \Br - 0.2111$ & 423&- & -  \\
         \hline
        513,945 mod 1944 & $\frac{1}{18} \log \Br - 0.1872$ & 179&- & -  \\
         \hline
        4644,1188 mod 7776 & $\frac{1}{18} \log \Br - 0.1872$ & 179&- & -  \\
         \hline
        Otherwise & $\frac{1}{18} \log \Br - 0.1196$ & 15&- & -  \\
        \hline
        
\end{tabular}
%\end{center}
\end{table}

Based on the data in Table \ref{c_p3}, to prove Lemma \ref{height}, in this case, we only need to check the height of rational points on the curves $E_{-21168}.$ 

The Mordell-weil group of  $E_{-21168}(\mathbb{Q})$ is a torsion-free group of rank 1, with $P=(84,756)$ as a generator. Using PARI, $\hat{h}(P)>\frac{1}{36} \log(21168)-0.1277 > \frac{1}{36}\log(21168) -0.1347 $.

This completes the proof of Lemma \ref{height}.

\subsection{Proof of Lemma \ref{heightcube}} \label{appendix2}
 In this subsection, we prove Lemma \ref{heightcube}. The proof is similar to the one in the previous subsection. Throughout this subsection, we assume $B$ is a cube. Since $B$ is a six-power free cube, the Tamagawa index at every prime $p>3$ is either 1, 2 or 4. Therefore only parts (a) and (b) of Theorem 1.2 of \cite{vou} is applicable here. Based on the possible congruence of $B$ modulo powers of 2 and 3, four possible cases that need to be examined.  
 
 \begin{itemize}
 
   \item[Case 1:]($B$ is congruent to 1 or 17 modulo 72.) Based on Theorem 1.2 and Tables 4 and 5 of \cite{vou}, if $E_B$ has a Tamagava index 1, at every prime $p>3$, then for any rational non-torsion point $P$ on $E_B$, we have $\hat{h}(P) > \frac{1}{6} \log \Br -0.2807$ and if the Tamagawa index is bigger than 1 in at least one prime $p>3$, then $\hat{h}(P) > \frac{1}{24} \log \Br -0.002$. For all integer values $\Br >9$, the following inequality holds:
  $$ \frac{1}{24} \log \Br -0.002 < \frac{1}{6} \log \Br -0.2807,$$
 and there is no cube congruent to 1 or 17 modulo 72 such that $E_B$ has a positive rank with $\Br \leq 9$.   
 
  \item[Case 2:]($B$ is congruent to 513 or 945 modulo 1944.) Based on Theorem 1.2 and Tables 4 and 5 of \cite{vou}, if $E_B$ has a Tamagava index 1, at every prime $p>3$, then for any rational non-torsion point $P$ on $E_B$, we have $\hat{h}(P) > \frac{1}{6} \log \Br -0.82996$ and if the Tamagawa index is bigger than 1 in at least one prime $p>3$, then $\hat{h}(P) > \frac{1}{24} \log \Br -0.002$. For all integer values $\Br >752$, the following inequality holds:
  $$ \frac{1}{24} \log \Br -0.002 < \frac{1}{6} \log \Br -0.82996,$$
 and there is no cube congruent to 513 or 945 modulo 1944 with $ \Br \leq 752$.
 
 \item[Case 3:]($B$ is congruent to 8 or 136 modulo 144.) Based on Theorem 1.2 and Tables 4 and 5 of \cite{vou}, if $E_B$ has a Tamagava index 1, at every prime $p>3$, then for any rational non-torsion point $P$ on $E_B$, we have $\hat{h}(P) > \frac{1}{6} \log \Br -0.3723$ and if the Tamagawa index is bigger than 1 in at least one prime $p>3$, then $\hat{h}(P) > \frac{1}{24} \log \Br +0.2290$. For all integer values $\Br >123$, the following inequality holds:
  $$ \frac{1}{24} \log \Br +0.2290 < \frac{1}{6} \log \Br -0.3723.$$
Therefore, we only need to check rational points on curves $E_8$ and $E_{-8}$. The Mordell curve $E_{-8}$ has rank 0. The The Mordell-weil group of  $E_{8}(\mathbb{Q})$ is group of rank 1 with a torsion subgroup isomorphic to $\mathbb{Z}/2 \mathbb{Z}$, with $P=(1,3)$ as a generator. Using PARI, $\hat{h}(P)>\frac{1}{24} \log(8)+0.23997 > \frac{1}{24}\log \Br +0.2290 $.

 \item[Case 4:]($B$ is congruent to 216 or 3672 modulo 144.) Based on Theorem 1.2 and Tables 4 and 5 of \cite{vou}, if $E_B$ has a Tamagava index 1, at every prime $p>3$, then for any rational non-torsion point $P$ on $E_B$, we have $\hat{h}(P) > \frac{1}{6} \log \Br -0.9455$ and if the Tamagawa index is bigger than 1 in at least one prime $p>3$, then $\hat{h}(P) > \frac{1}{24} \log \Br +0.2290$. For all integer values $\Br >12040$, the following inequality holds:
  $$ \frac{1}{24} \log \Br +0.2290 < \frac{1}{6} \log \Br -0.9455.$$
Therefore, we only need to check rational points on curves $E_{216}$ and $E_{-216}$. The Mordell curve $E_{216}$ has rank 0. The Mordell-weil group of  $E_{-216}(\mathbb{Q})$ is group of rank 1 with a torsion subgroup isomorphic to $\mathbb{Z}/2 \mathbb{Z}$, with $P=(10,28)$ as a generator. Using PARI, $\hat{h}(P)>\frac{1}{24} \log(8)+0.7558 > \frac{1}{24}\log \Br +0.2290 $.

This completes the proof.

\end{itemize}
\end{document}